\documentclass[11pt]{article}
\usepackage{amsmath,amssymb,amstext,dsfont,fancyvrb,float,fontenc,graphicx,subfigure,theorem,hyperref}
\usepackage[utf8]{inputenc}
\usepackage{hyperref}
\usepackage{url}
\usepackage{color}
\usepackage{breakurl}
\newcommand{\bburl}[1]{\textcolor{blue}{\url{#1}}}
\usepackage{comment}
\allowdisplaybreaks

\newcommand{\Z}{\ensuremath{\mathbb{Z}}}
\newcommand{\Q}{\mathbb{Q}}

\newcommand{\twocase}[5]{#1 \begin{cases} #2 & \text{{\rm #3}}\\ #4
&\text{{\rm #5}} \end{cases}   }

\newcommand\be{\begin{equation}}
\newcommand\ee{\end{equation}}
\newcommand\bea{\begin{eqnarray}}
\newcommand\eea{\end{eqnarray}}
\newcommand{\ncr}[2]{{#1 \choose #2}}

\usepackage{pgfplots}
\usepackage{graphicx}
\newcommand{\js}[1]{{#1\overwithdelims () p}}
\usepackage{tikz}
\usepackage{longtable}

\usepackage{tikz,everypage}
\AtBeginDocument{%
  \AddEverypageHook{%
    \begin{tikzpicture}[remember picture,overlay]
    \end{tikzpicture}%
  }%
}
\tikzset{%
  line numbers/.store in=\fakelinenos,
  line numbers=50,
  line number shift/.store in=\fakelinenoshift,
  line number shift=5mm,
line number style/.style={text=black},
}

\usepackage[letterpaper]{geometry}
\ifx\volno\undefined\def\volno{0}\fi
\ifx\volyear\undefined\def\volyear{2017}\fi
\ifx\pagno\undefined\def\pagno{000--000}\fi

\newfont{\footsc}{cmcsc10 at 8truept}
\newfont{\footbf}{cmbx10 at 8truept}
\newfont{\footrm}{cmr10 at 10truept}

\usepackage{fancyhdr}
\usepackage{relsize}
\usepackage{sectsty}
\allsectionsfont{\larger[-1]}

\renewcommand\paragraph{\@startsection{paragraph}{4}{\z@}
                                    {2ex \@plus.5ex \@minus.2ex}
                                    {-1em}
                                    {\normalfont\normalsize\bfseries}}

\renewcommand\subparagraph{\@startsection{subparagraph}{5}{\parindent}
                                       {2ex \@plus.5ex \@minus .2ex}
                                       {-1em}
                                      {\normalfont\normalsize\bfseries}}

\newlength{\BiblioSpacing}
\renewenvironment{thebibliography}[1]{
\begin{oldthebibliography}{#1}
\setlength{\parskip}{\BiblioSpacing}
\setlength{\itemsep}{\BiblioSpacing}
}
{
\end{oldthebibliography}
}

\usepackage[strict]{changepage}
\def\abstractname{Abstract -}   
\def\abstract{\begin{adjustwidth}{1cm}{1cm} \par    \footnotesize \noindent {\bf \abstractname}
\def\endabstract{ \end{adjustwidth} \smallskip }}

{\theorembodyfont{\itshape}\newtheorem{theorem}{Theorem}[section]}
{\theorembodyfont{\itshape}}
{\theorembodyfont{\itshape}\newtheorem{definition}[theorem]{Definition}}
{\theorembodyfont{\itshape}\newtheorem{lemma}[theorem]{Lemma}}
{\theorembodyfont{\itshape}}
{\theorembodyfont{\rm}}
{\theorembodyfont{\rm}}
{\theorembodyfont{\rm}\newtheorem{conjecture}[theorem]{Conjecture}}
{\theorembodyfont{\rm }}

\title{\Large\bf Biases in Moments of the Dirichlet Coefficients in One- and Two-Parameter Families of Elliptic Curves}
\author{\sc S. J. Miller, Y. Weng}

\begin{document}
\setcounter{page}{1}
\maketitle
\thispagestyle{fancy}

\vskip 1.5em

\begin{abstract}
Elliptic curves arise in many important areas of modern number theory. One way to study them is take local data, the number of solutions modulo $p$, and create an $L$-function. The behavior of this global object is related to two of the seven Clay Millenial Problems: the Birch and Swinnerton-Dyer Conjecture and the Generalized Riemann Hypothesis. We study one-parameter families over $\Q(T)$, which are of the form $y^2=x^3+A(T)x+B(T)$, with non-constant $j$-invariant. We define the $r$\textsuperscript{th} moment of an elliptic curve to be $A_{r,E}(p) :=  \frac1{p} \sum_{t \bmod p} a_t(p)^r$, where $a_t(p)$ is $p$ minus the number of solutions to $y^2 = x^3 + A(t)x + B(t) \bmod p$. Rosen and Silverman showed biases in the first moment equal the rank of the Mordell-Weil group of rational solutions.

Michel proved that $pA_{2,E}(p)=p^2+O(p^{3/2})$. Based on several special families where computations can be done in closed form, Miller in his thesis conjectured that the largest lower-order term in the second moment that does not average to $0$ is on average negative. He further showed that such a negative bias has implications in the distribution of zeros of the elliptic curve $L$-function near the central point. To date, evidence for this conjecture is limited to special families. In this paper, we explore the first and second moments of some one- and two-parameter families of elliptic curves, looking to see if the biases persist and exploring the consequence these have on fundamental properties of elliptic curves. We observe that in all of the one- and two-parameter families where we can compute in closed form that the first term that does not average to zero in the second-moment expansion of the Dirichlet coefficients has a negative average.  In addition to studying some additional families where the calculations can be done in closed form, we also systematically investigate families of various rank. These are the first general tests of the conjecture; while we cannot in general obtain closed form solutions, we discuss computations which support or contradict the conjecture. We then generalize to higher moments, and see evidence that the bias continues in the even moments.
\end{abstract}

\tableofcontents


\begin{MSC}
60B10, 11B39, 11B05  (primary) 65Q30 (secondary)\end{MSC}

\section{Introduction}

The distribution of rational points on elliptic curves are not just of theoretical interest, but also have applications in encryption schemes. While it is often difficult to study one particular curve, frequently great progress can be made by looking at families of curves and computing averages. One powerful tool for such calculations are the associated $L$-functions. In particular, negative biases in the first moment of their Dirichlet coefficients are known in many cases (and conjectured in general) to be related to the rank. Recent investigations suggest that the second moment has similar biases, and these have applications to the distribution of the zeros of their $L$-functions.

We report on some of these calculations, as well as extensions to higher moments. We deliberately take a leisurely approach to make this paper reasonably self-contained, motivating the history and background material as problems of this nature are accessible with minimal pre-requisites (for more details on elementary number theory, see for example \cite{Da1, Da2, MT-B, NZM}). Our hope is to encourage others to continue these investigations in related families. For those interested in the code, email the authors.

\subsection{Rational Points on a Quadratic Equation}

For thousands of years, there has been interest in finding integer solutions to equations or systems of equations with integer coefficients. These are called Diophantine equations' perhaps the most famous is the Pythagorean theorem.

\begin{theorem} [Pythagorean Theorem]
    If $a$ and $b$ are the sides of a right triangle with hypotenuse $c$, then
\begin{eqnarray}
a^2+b^2\ = \ c^2.
\end{eqnarray}
\end{theorem}

However, it is not immediately clear that there are any rational solutions, though a search quickly finds many. It distressed the Greeks that the right triangle with sides of integer length 1 and 1 has a hypotenuse of irrational length $\sqrt{2}$. By rescaling a rational triple we may assume that the sides are integral and relatively prime; we call such primitive Pythagorean triples, and can write down an explicit formula to generate all Pythagorean triples.


\begin{lemma}[Pythagorean Triples] Given any Pythagorean triple there exist positive integers $k$, $m$ and $n$ with $m > n$ such that
\begin{eqnarray}
a\ =\ k\cdot (m^{2}-n^{2}),\ \ \ \,b\ = \ k\cdot (2mn),\ \ \ \,c\ = \ k\cdot (m^{2}+n^{2}),
\end{eqnarray} where $m$ and $n$ are coprime and not both odd.
\end{lemma}

\begin{figure}[h]
\begin{minipage}[c]{0.4\linewidth}
\tikzset{every picture/.style={line width=0.5pt}}
\begin{tikzpicture}[scale=0.75,x=0.75pt,y=0.75pt,yscale=-1,xscale=1]
\draw   (101,100) -- (280.5,226) -- (101,226) -- cycle ;
\draw (82,164) node  [align=left] {a};
\draw (166,241) node  [align=left] {b};
\draw (200,145) node  [align=left] {c};
\end{tikzpicture}
\caption{A right triangle with side length of a, b and c}
\label{triangle}
\end{minipage}
\hfill
\begin{minipage}[c]{0.5\linewidth}
\begin{tikzpicture}[scale=1.5,cap=round,>=latex,baseline={(0,0)}]
    \draw[->] (-1.5cm,0cm) -- (1.5cm,0cm) node[right,fill=white] {$x$};
    \draw[->] (0cm,-1.5cm) -- (0cm,1.5cm) node[above,fill=white] {$y$};
    \draw[thick] (0cm,0cm) circle(1cm);
        \draw (-1.25cm,0cm) node[above=1pt] {$(-1,0)$}
    (1.25cm,0cm)  node[above=1pt] {$(1,0)$}
    (0cm,-1.25cm) node[fill=white] {$(0,-1)$}
    (0cm,1.25cm)  node[fill=white] {$(0,1)$};
\draw    (-1,-0) -- (0.8,0.6) ;
\draw    (0,-0) -- (0.8,0.6) ;
\fill[black] (-1,0) circle (0.5mm);
\fill[black] (0.8,0.6) circle (0.5mm)
node [above right] {$(x,y)$};
\fill[black] (0,0.33) circle (0.5mm)
node [above left] {$(0, t)$};
\end{tikzpicture}
\caption{A rational parametrization of the circle $x^2+y^2=1$}
\label{Circle}
\end{minipage}%
\end{figure}

\begin{proof} Finding integer Pythagorean triples (see Figure \ref{triangle}) is equivalent to finding rational points on the unit circle $x^2+y^2=1$; just let
\begin{eqnarray} \label{1.3}
  x \ = \ \frac{a}{c} \ \ \ \text{and} \ \ \ y \ = \ \frac{b}{c}.
\end{eqnarray}

We now find all rational points on the unit circle. Let $(x, y)$ denote an arbitrary point on the circle. In Figure \ref{Circle}, we know one rational solution\footnote{There are three other obvious rational solutions which we could have used; the standard convention is to use this one.}, $(-1, 0)$. The line through $(x, y)$ with slope $t$ is given by the equation
\begin{eqnarray}
y \ = \ t(1+x).
\end{eqnarray}
Hence, the other point of intersection of the line with the unit circle is
\begin{eqnarray}
1-x^2 \ = \ y^2 \ =\  t^2(1+x)^2.
\end{eqnarray}
Dividing each side by the root $(1+x)$, corresponding to the root $x=-1$, we get
\begin{eqnarray}
1-x \ = \ t^2(1+x).
\end{eqnarray}
Using the above relation, we find
\begin{eqnarray} \label{1.7}
x\ = \ \frac{1-t^2}{1+t^2} \ \ \ \ y \ = \ \frac{2t}{1+t^2}.
\end{eqnarray}
Thus if $x$ and $y$ are rational numbers, then the slope $t=y/(1+x)$ is also rational. Conversely, if $t$ is rational then $x$ and $y$ are rational. Hence, by letting $t$ range over the rational numbers numbers, we generate all the rational pairs on the circle (except $(-1,0)$ as in this case $t$ is infinite).
\end{proof}

Since we are able to generate the rational points on a quadratic equation, it is natural to study how to generate the rational points on a cubic equation, such as an elliptic curve.

\subsection{Introduction to Elliptic Curves}

We motivate studying elliptic curves in general by investigating first special cases, arising from right triangles
with area 1. We have the following equation:
\begin{eqnarray}
  1 \ = \  \frac{1}{2}ab.
\end{eqnarray}
We substitute $a=xc$ and $b=yc$ from (\ref{1.3}) and obtain
\begin{eqnarray}
  1 \ = \  \frac{1}{2}c^2xy.
\end{eqnarray}
Plugging in our results from (\ref{1.7}) gives
\begin{eqnarray}
  1 & \ = \ &  \frac{1}{2}c^2\left(\frac{1-t^2}{t^2+1}\right)\left(\frac{2t}{t^2+1}\right) \nonumber\\
   & \ = \ & \frac{c^2}{(t^2+1)^2}(t-t^3).
\end{eqnarray}
Divided both sides by $c^2/(t^2+1)^2$, we get
\begin{eqnarray}
    \left(\frac{t^2+1}{c}\right)^2 \ = \ t-t^3.
\end{eqnarray}
If we let $Y=(t^2+1)/c$ and $X=-t$ we obtain
\begin{eqnarray}
    Y^2 \ = \ X^3-X,
\end{eqnarray} which is an equation of an elliptic curve, which we formally define; see \cite{ST} for more details.


\begin{figure}[h]
\begin{center}
\end{center}
\end{figure}
\vspace{-5mm}

An elliptic curve in standard (or Weierstrass) form is the set of points $(x, y)$ satisfying the cubic equation
\begin{eqnarray}
    y^2\ =\ x^3 + ax + b,
\end{eqnarray} where $a,b\in\mathbb{Q}$ and the discriminant $4a^3+27b^2$ is not zero. This last condition is to avoid degenerate cases. For example,
we do not want $y^2 = x^2 (x-1)$ to be an elliptic curve; when we send $y$ to $xy$ we get $y^2 = x-1$, a parabola. More generally, it is of the form \be y^2 + a_1 x y + a_3 y \ = \ x^3 + a_2 x^2 + a_4 x + a_6. \ee
We study two kinds of families of elliptic curves: one-parameter and two-parameter. While we can change variables to put our curves in standard form, for convenience we often have an $x^2$ term. In the definitions below if we specialize the variables to integers we obtain an elliptic curve (provided of course that the discriminant is non-zero).

\begin{definition}[One-Parameter Family of Elliptic Curves] A one-parameter family is of the form
\begin{eqnarray}
\mathcal{E}:y^2\ = \ x^3 + A(T)x + B(T),
\end{eqnarray} with $A(T),B(T)\in\mathbb{Q}[T]$, which are polynomials of finite degree and rational coefficients.
\end{definition}

\begin{definition}[Two-Parameter Family of Elliptic Curves] A two-parameter family is of the form
\begin{eqnarray}
y^2 \ = \ x^3 + A(T, S)x + B(T, S),
\end{eqnarray} with $A(T, S),B(T, S)\in\mathbb{Q}[T,S]$.
\end{definition}

One of the reasons that there is such interest in elliptic curves is the following result.

\begin{theorem}[Mordell's Theorem]
The set of rational points on an elliptic curve is a finitely generated group.
\end{theorem}

\begin{figure}[h]
\begin{minipage}[c]{0.4\linewidth}
\includegraphics[width=0.8\linewidth, height=70px]{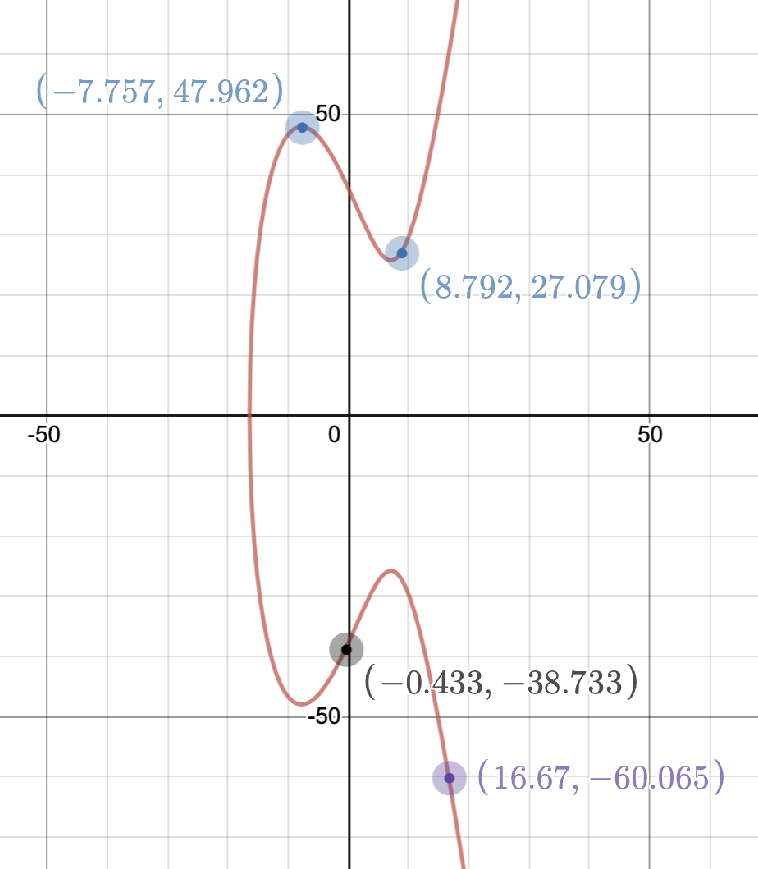}
  \caption{Points within the range $|x| \leq 20$ on Rank 0 Elliptic Curve $E: y^2=x^3+x^2-165x+1427$.}
  \label{fig:mordellone}
\end{minipage}
\hfill
\begin{minipage}[c]{0.4\linewidth}
\includegraphics[width=0.8\linewidth, height=70px]{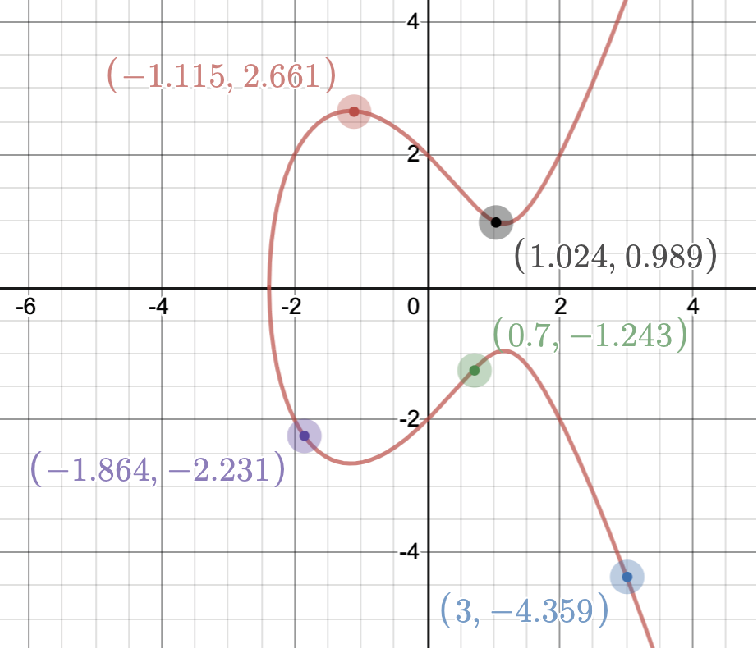}
      \caption{Points within the range $|x| \leq 20$ on Rank 1 Elliptic Curve $E: y^2=x^3-4x+4$.}
\label{fig:mordelltwo}
\end{minipage}%
\end{figure}

Figures \ref{fig:mordellone} and \ref{fig:mordelltwo} demonstrate the addition law for an elliptic curve of rank 0 and an elliptic curve of rank 1; the ``point at infinity'' acts as the identity element for addition. As the rank of the elliptic curve increases, there are typically more points within a certain range of $x$.

RSA cryptography, which is based on groups arising from primes or the product of two distinct primes $p$ and $q$, ($\mathbb{Z}/pq\mathbb{Z}$), was the gold standard in cryptography for years. However, it was also well-known that if we are able to factor a large number, then we can easily break RSA. Hence, it led to a search for other interesting groups with more complicated structure. Elliptic curves became the natural candidate because they have a group structure. Two points generate a third, but note that for the Pythagorean triples we only needed to find one point to generate them all. See \cite{RG} for more details.

Next, we define a characteristic of elliptic curves that is relevant to our paper. Often one can gain an understanding of a global object by studying a local one. In particular, for a prime $p$ we can look at how often we have pairs $(x,y)$ satisfying $y^2 = x^3 + ax + b \bmod p$. As half of the non-zero elements of $\mathbb{F}_p = \mathbb{Z}/p\mathbb{Z} = \{0, 1, 2, \dots, p-1\}$ are non-zero squares modulo $p$ and the other half are not squares, it is reasonable to expect that for a randomly chosen $x$ that half the time it will generate two solutions modulo $p$ and half the time it will generate zero. Thus we expect the number of pairs to be of size $p$, and it is valuable to look at fluctuations about this expected number.

\begin{definition}[Dirichlet Coefficients]  For $E$ an elliptic curve $y^2 = x^3 + ax + b$ and a prime $p$, we define the Dirichlet coefficients $a_E(p)$ by
\begin{eqnarray}
a_E(p) \ := \ p - |E(\mathbb{F}_p)|,
\end{eqnarray}
where $|E(\mathbb{F}_p)|$ is the number of solutions $(x,y)$ to $y^2 = x^3 + ax + b \bmod p$ with $x, y \in \mathbb{F}_p$. These are used in constructing the associated $L$-function to the elliptic curve, $L(E,s) = \sum_n a_E(n)/n^s$, which generalizes the Riemann zeta function $\zeta(s) = \sum_n 1/n^s$.
\end{definition}

There is a very useful formula for $a_E(p)$ (if the curve $E$ is clear we often suppress the subscript adn write $a(p)$ or $a_p$). The Legendre symbol $\js{a}$ is zero if $a$ is zero modulo $p$, 1 if $a$ is a non-zero square modulo $p$, and $-1$ otherwise. Thus $1 + \js{x^3+ax+b}$ is the number of solutions modulo $p$ for a fixed $x$. If we sum this over all $x$ modulo $p$ we obtain $|E(\mathbb{F}_p)|$, and thus
\begin{eqnarray}\label{eq:jacobisumec}
a_E(p) \ = \ -\sum_{x \bmod p} \js{x^3 + ax + b}.
\end{eqnarray}


Much is known about the $a(p)$'s. We focus on their size and average behavior, though recent breakthroughs have determined much more about their distribution.

\begin{theorem}[Hasse, 1931]
The Riemann Hypothesis for finite fields holds if $E$ is an elliptic curve and $p$ a prime; we have
\begin{eqnarray}
    |a_{E}(p)|\ \leq\ 2\sqrt{p}.
\end{eqnarray}
\end{theorem}

Hasse's theorem is very similar to the Central Limit Theorem, which itself is an example of the philosophy of square-root cancelation: if we have $N$ objects of size 1 with random signs, then frequently the sum is of size 0 with fluctuations on the order of $\sqrt{N}$.\footnote{Results such as these are correct up to the power of $N$ but can miss logarithms; } In our setting, we expect half of the time $\js{x^3 + A(t)x + B(t)}$ equals 1, and the other half time it is -1. As we have $p$ terms of size 1 with random signs, the results should be of size $\sqrt{p}$, which is Hasse's theorem. We often use big-Oh notation to denote sizes of the quantities we study.

\begin{definition}[Big-Oh Notation] We say $f = O(g(x))$, read $f$ is big-Oh of $g$, if there exists an $x_0$ and a $B > 0$ such that for all $x \ge x_0$ we have $|f(x)| \le B g(x)$.
\end{definition}


Last but not least, we define some other important characteristics of elliptic curves. As the $a_{\mathcal{E}_t}(p)$ only depend on $t \bmod p$ by \eqref{eq:jacobisumec}, for a fixed prime $p$ we only need to study specializations of $T$ modulo $p$.

\begin{definition}[Moment of a One-Parameter Family] Let $\mathcal{E}$ be a one parameter family of elliptic curves $y^2 = x^3 + A(T)x + B(T)$ over $\mathbb{Q}(T)$, with $\mathcal{E}_t$ the specialized curves. For each positive integer $r$, we define the r\textsuperscript{{\rm th}} moment:
\begin{eqnarray}
    A_{\mathcal{E},r(p)}\ :=\ \frac{1}{p}\sum\limits_{t\bmod p}a_{\mathcal{E}_t}(p)^r.
\end{eqnarray}
\end{definition}

There is a natural extension to two-parameter families, where we sum over $s$ and $t$ modulo $p$.

We conclude with one final concept and result, the rank.

\begin{definition}[Geometric Rank of $E$] The group of rational solutions of an elliptic curve $E$, denoted $E(\mathbb{Q})$, can be written as $r$ copies of $\Z$ plus a finite torsion part:
\begin{eqnarray}
E(\mathbb{Q}) \ \cong\ \mathbb{Z}^r \times E(\mathbb{Q})_{\rm tors};
\end{eqnarray} $r$ is the geometric rank of $E$.
\end{definition}

The analytic rank of $E$ is the order of vanishing of the associated $L$-function at the central point. Similar to many other problems in mathematics, frequently one of these objects is easier to study than the other, and the hope is that there is a connection between them. This is true for elliptic curves; we turn to this next, and see the key role played by the $a_E(p)$'s.

\subsection{From Random Matrix Theory to the Birch and Swinnerton-Dyer Conjecture}

The Birch and Swinnerton-Dyer conjecture is one of the seven Clay Millenial Problems; these were formulated in the spirit of Hilbert's successful list from the start of the twentieth century, and are meant to inspire and highlight important mathematics. It is based on a $L$-function of an elliptic curve, connecting analysis to geometry, two great different fields of mathematics. Before stating it, we first describe some of the problems and methods of modern number theory to motivate both why we care about this conjecture, as well as the main topic of this paper. For more on this story, see \cite{BFMT-B, FM}.

Given the inability to theoretically describe the energy levels of atoms more complicated than hydrogen, due to the complexities of the mathematics, physicists developed statistical approaches. Based on extensive numerical data, Wigner proposed that one could model nuclear physics through random matrices; that the behavior of eigenvalues in these matrix ensembles described the behavior of energy levels. Briefly, from quantum mechanics we can write down the equation \be H \psi_n \ = \ E_n \psi_n, \ee where $H$ is the Hamiltonian, $\psi_n$ are the energy eigenstates and $E_n$ the energy levels. Unfortunately the nuclear forces leading to the operator $H$ are too complicated to allow us to write it down and analyze it explicitly. Wigner's great insight was to instead consider a probability distribution on $N \times N$ matrices, calculate averages over these ensembles, scale them appropriately so that limiting behavior exists as $N\to\infty$, and the appeal to a central limit theorem type of law to show that a typical operator will have behavior close to the system average. These predictions were confirmed experimentally in studies of the energy levels of heavy nuclei.

Amazingly, similar results were found in the spacings between the zeroes of the Riemann Zeta function, which connects integers to primes and helps us understand the mysterious distribution of primes, seem to follow the RMT prediction too.

\begin{definition} [Riemann Zeta Function] For ${\rm Re}(s) > 1$
\begin{eqnarray}\zeta(s) \ :=\ \sum_{n=1}^{\infty} \frac{1}{n^s}\ =\ \prod_{p\ {\rm prime}}\left(1 - \frac1{p^s}\right)^{-1}.
\end{eqnarray}
\end{definition}

The zeta function is defined as the sum over integers above, but its utility comes from the product expansion (which follows immediately from the geometric series formula and the fundamental theorem of arithmetic, which states each integer can be written uniquely as a product of prime powers in increasing order). Initially defined only for ${\rm Re}(s) > 1$, the zeta function can be analytically continued to the entire complex plane with a simple pole of residue 1 at $s=1$:
\begin{eqnarray}
\xi(s)\ :=\ \Gamma\left(\frac{s}{2}\right)\pi^{-\frac{s}{2}} \zeta(s)\ =\ \xi(1-s),
\end{eqnarray}
where the line ${\rm Re}(s) = 1/2$ is the critical line, and $s=1/2$ the central point. The Riemann hypothesis states all the non-trivial zeros of $\zeta(s)$ have real part equal to 1/2 (due to the presence of the Gamma factor, $\zeta(s)$ vanishes at the negative even integers). By doing a contour integral of the logarithmic derivative of $\zeta(s)$ and shifting contours, one obtains the Explicit Formula, which relates a sum over zeros to a sum over prime. Figures \ref{fig:swave} and \ref{fig:odlyzko} show similar behavior in the spacings between energy levels of heavy nuclei and spacings between zeros of $\zeta(s)$.

\begin{figure}[h]
\begin{minipage}[c]{0.4\linewidth}
\includegraphics[width=1\linewidth, height=100px]{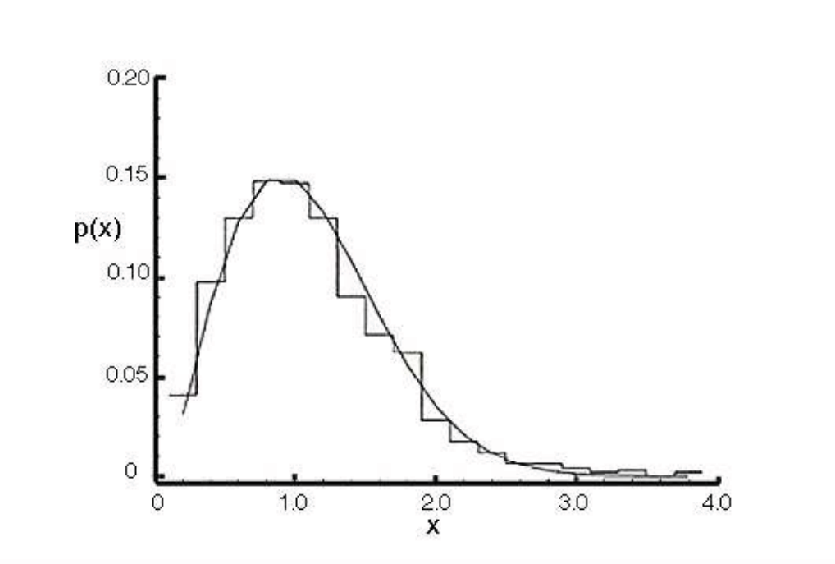}
 \caption{ A Wigner distribution fitted to the spacing distribution of 932 s-wave resonances in the interaction Uranian + n at energies up to 20 keV.}
\label{fig:swave}
\end{minipage}
\hfill
\begin{minipage}[c]{0.4\linewidth}
\includegraphics[width=0.8\linewidth, height=100px]{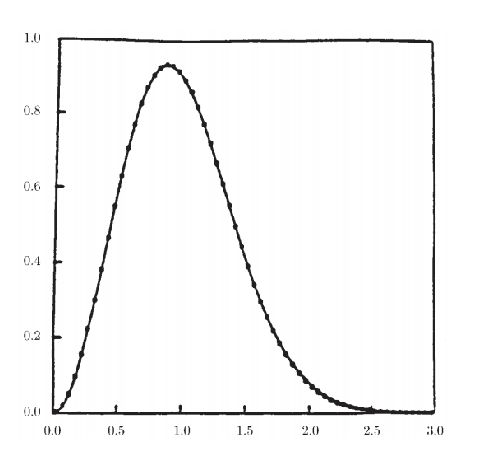}
 \caption {70 million spacings between adjacent zeros of $\zeta(s)$, starting at the $10^{20}$\textsuperscript{th} zero. The solid curve is the RMT prediction for the GUE ensemble, and the dots are the zeta zeros (from Odlyzko).}
 \label{fig:odlyzko}
\end{minipage}%
\end{figure}


We now move on to discuss other $L$-functions; for more on these see for example \cite{Kn}. With the normalization below, the critical strip is $0 < {\rm Re}(s) < 2$, and the functional equation of the completed elliptic curve $L$-function relates values at $s$ to those at $2-s$.

\begin{definition}[$L$-function]
The Hasse-Weil $L$-function of an elliptic curve $E: y^2 + a_1xy + a_3y = x^3 + a_2x^2 + a_4x +a_6$ with coefficient $a_{E}(p)$ and discriminant $\Delta$,
\begin{eqnarray}
\Delta\ :=\ {-b_2}^2{b_8}-8{b_4}^3-27{b_6}^2+9{b_2}{b_4}{b_6},
\end{eqnarray}
where ${b_2}={a_1}^2+4a_4$, $b_4=2a_4+a_1a_3$ and $b_6={a_3}^2+4{a_6}$, is defined as
\begin{eqnarray}
L(s,E) & \ :=\ & \prod_{p|\Delta} \frac{1}{1 - a_p p^{-s}} \prod_{p
\nmid \Delta} \frac{1}{1 - a_p p^{-s} + p^{1-2s}}.
\end{eqnarray}
\end{definition}

Similar to the zeta function, these $L$-functions take local data and create a global object, from which much can be deduced; in Appendix \ref{appendix} we give an example of how we can piece together local information to construct a global object which, if a closed form expression exists, can allow us to deduce information about the objects of interest. The most important of these inferences is the famous Birch and Swinnerton-Dyer conjecture.


\begin{conjecture}[Birch and Swinnerton-Dyer Conjecture] The order of vanishing of\newline $L(E,s)$ at the central point $s=1$ is equal to the rank of the group of rational points $E(\mathbb{Q})$.
\end{conjecture}

In other words, Birch and Swinnerton-Dyer conjectured that the geometric rank of an elliptic curve equals its analytic rank.


Unfortunately, it is not known what values of rank $r$ are possible for an elliptic curve. In 1938, Billing found an elliptic curve with rank $3$. The largest known rank increased over the next few decades. The largest is due to Elkies in 2006, and is rank at least $28$. Interestingly, there are not examples of elliptic curves for each rank smaller than $28$ (see \cite{Du} for a more comprehensive historical data on elliptic curve records). While originally it was thought that the ranks are unbounded, now some conjecture that this is not the case.
\vspace{-2.5mm}
\begin{center}
\begin{tikzpicture}[scale=0.8]
\begin{axis}[
x tick label style={/pgf/number format/.cd,%
          scaled x ticks = false,
          set thousands separator={},
          fixed},
    title={Elliptic Curve Records},
    xlabel={Year},
    ylabel={Rank $>=$},
    xmin=1930, xmax=2019,
    ymin=0, ymax=30,
    xtick={1930, 1950, 1970, 1990, 2019},
    ytick={0,5,10,15,20,25,30},
    ymajorgrids=true,
    grid style=dashed,
]
\addplot+[const plot, no marks, thick] coordinates {(1938, 3)(1945,4)(1974,6)(1975,7)(1977,8)(1977,9)(1982,12)(1986,14)(1992,15)(1992,17)(1992,19)(1993,20)(1994,21)(1997,22)(1998,23)(2000,24)(2006,28)(2019,28)} node[below=1.15cm,pos=.76,black] {\footnotesize};
\end{axis}
\end{tikzpicture}
\end{center}

\subsection{The Bias Conjecture}

We are now ready to state our main object of study, the bias conjecture. The original motivation for it comes from the distribution of low-lying zeros in families of $L$-functions; this is part of the $n$-level densities introduced by Katz and Sarnak \cite{KS1, KS2}. The next few paragraphs are thus more technical and assume some familiarity of the subject, and may be safely skipped.

Similar to using the Riemann Zeta function to understand the distribution of primes, we use the Explicit Formula, which relates sums over primes of the Dirichlet coefficients $a_E(p)$ and $a_E^2(p)$ to sums of test functions over zeros, to deduce information about the zeros. We look at a one-parameter family $\mathcal{E}: y^2=x^3 + A(t)x + B(t)$, with $t \in [N, 2N]$, and where $\phi$ is an even Schwartz-class function that decays rapidly (this means $\phi$, and all of its derivatives, decay faster than $1/(1+|x|)^A$ for any $A>0$), $\log R$ is the average log conductor (and tells us how to scale the zeros near the central point $s=1$), and $1+i\gamma$ are the non-trivial zeros of the $L$-function:
\begin{align}
& \frac1N \sum_{t=N}^{2N}\sum_{\gamma_t}
\phi\left(\gamma_t\frac{\log R}{2\pi}\right) \ = \
\widehat{\phi}(0) + \phi(0) - \ \frac2N \sum_{t=N}^{2N} \sum_p
\frac{\log p}{\log R} \frac1{p} \widehat{\phi}\left(\frac{\log p}{\log
R} \right)
a_t(p) \nonumber\\
& \ \ \ - \ \frac2N \sum_{t=N}^{2N} \sum_p \frac{\log p}{\log R}
\frac{1}{p^2} \widehat{\phi}\left(\frac{2\log p}{\log R} \right)
a_t(p)^2 + O\left(\frac{\log\log R}{\log R}\right);
\end{align}
the result above comes from integrating the logarithmic derivative of the $L$-function against the Schwartz test function $\phi$ and then shifting contours. If the generalized Riemann Hypothesis is true then $\gamma \in \mathbb{R}$.

Note that if the test function is non-negative, then dropping the contributions of $\phi$ at all the zeros that are not at the central point removes a non-negative amount from the left hand side. The right hand side then becomes an upper bound for the average rank of the elliptic curves in the family:
\begin{align}
& \frac1N \sum_{t=N}^{2N}\sum_{\gamma_t = 0}
\phi\left(0\right) \ \le \
\widehat{\phi}(0) + \phi(0) - \ \frac2N \sum_{t=N}^{2N} \sum_p
\frac{\log p}{\log R} \frac1{p} \widehat{\phi}\left(\frac{\log p}{\log
R} \right)
a_t(p) \nonumber\\
& \ \ \ - \ \frac2N \sum_{t=N}^{2N} \sum_p \frac{\log p}{\log R}
\frac{1}{p^2} \widehat{\phi}\left(\frac{2\log p}{\log R} \right)
a_t(p)^2 + O\left(\frac{\log\log R}{\log R}\right),
\end{align}  \normalsize
which means that
\begin{align}
&
\phi\left(0\right) \ast {\rm AverageRank}(N) \ \le \
\widehat{\phi}(0) + \phi(0) - \ \frac2N \sum_{t=N}^{2N} \sum_p
\frac{\log p}{\log R} \frac1{p} \widehat{\phi}\left(\frac{\log p}{\log
R} \right)
a_t(p) \nonumber\\
& \ \ \ - \ \frac2N \sum_{t=N}^{2N} \sum_p \frac{\log p}{\log R}
\frac{1}{p^2} \widehat{\phi}\left(\frac{2\log p}{\log R} \right)
a_t(p)^2 + O\left(\frac{\log\log R}{\log R}\right).
\end{align}  \normalsize

Thus when $\phi$ is non-negative, we obtain a bound for the
average rank in the family by restricting the sum to be only over zeros at the central point. The error $O\left(\log \log R/\log R\right)$ comes from trivial estimation and ignores probable cancelation, and we expect $O\left(1/\log R\right)$ or smaller to be the correct magnitude. For most one-parameter families of elliptic curves we have $\log R \sim \log N^a$ for some integer $a$, where $t \in [N, 2N]$.

The main term of the first and second moments of the $a_t(p)$ give $\phi\left(0\right) \ast {\rm AverageRank}(N)$ and $-\frac12 \phi(0)$; this is a standard application of the prime number theorem to evaluate the resulting sums; for details see the appendices on prime sums in \cite{Mi1}. This is reminiscent of the Central Limit Theorem, where so long as some weak conditions are satisfied for independent, identically distributed random variables, their normalized sum converges to the standard normal. In that setting, if the moments are finite we can always adjust the distribution to have mean zero and variance one, and it is only these moments that enter the limiting analysis. The higher moments \emph{do} have an impact, but it is only through the lower order terms, which control the \emph{rate} of convergence.

We have a similar situation here. First, the higher moments of the Dirichlet coefficients contribute in the big-Oh terms $O\left(1/\log R\right)$. Second, the lower order terms in the first and second moments can contribute, but not to the main term in the expansions above. Explicitly, assume the second moment of $a_t(p)^2$ is $p^2 - m_\mathcal{E}p + O(1)$, $m_\mathcal{E} > 0$. We have already handled the contribution from $p^2$, and $-m_\mathcal{E}p$ contributes
\begin{eqnarray}
S_2 &\ \sim\ & \frac{-2}{N} \sum_p \frac{\log p}{\log
R}\widehat{\phi}\left(2\frac{\log p}{\log R}\right) \frac{1}{p^2}
\frac{N}{p}(-m_{\mathcal{E}}p) \nonumber\\ &=&
\frac{2m_{\mathcal{E}}}{\log R} \sum_p
\widehat{\phi}\left(2\frac{\log p}{\log R} \right) \frac{\log
p}{p^2}.
\end{eqnarray}

We thus have a prime sum which converges, and this sum is bounded by $\sum_p \log p/{p^2}$. Thus, $S_2$ converges and there is a contribution of size $1/\log(R)$. This is the motivation behind why the Bias conjecture, which S. J. Miller conjectured in his thesis \cite{Mi1}, matters, as a bias has an impact in our estimates on the rank and the behavior of zeros near the central point.

\begin{conjecture}[Second Moment Elliptic Curve Bias Conjecture] Consider a family of elliptic curves. Then the largest lower term in the second moment expansion of a family which does not average to $0$ is on average negative. \end{conjecture}

If the Bias conjecture holds, then when we estimate the rank of a family, there is always an extra term that slightly increases the upper bound for the average rank. This amount decreases as $\log R$ grows, and thus in the limit plays no role; however, it does lead to a small but noticeable contribution for small and modest sized conductors.

\subsection{Our results}

We report on our results. Much is known about the first moment of the Dirichlet coefficients of elliptic curves. Work of Nagao \cite{Na1, Na2} and Rosen and Silverman \cite{RoSi} shows that the first moment in families is related to the rank of the family over $\mathbb{Q}(T)$; specifically, a small negative bias results in rank. This was used by Arms, Lozano-Robledo and Miller \cite{ALM} to construct one-parameter families of elliptic curves with moderate rank, and later generalized to elliptic curves over number fields \cite{MMRSY} and hyper-elliptic curves \cite{HKLL-RM}.

It is thus natural to ask if there is a bias in the second moments, and if so what are the consequences. We have already seen that a negative bias here is related to some of the observed excess rank and repulsion of zeros of elliptic curve $L$-functions near the central point for finite conductors.

We start with a result from Michel \cite{Mic} on the main term of the second moments, and the size of the fluctuations, in one-parameter families.

\begin{theorem}
For a one-parameter family $\mathcal{E}: y^2=x^3+A(T)x+B(T)$ with non-constant $j(T)$-invariant $j(T)=1728\frac{4A(T)^3}{4A(T)^3+27B(T)^2}$, the second moment of the Dirichlet coefficients equals
\begin{eqnarray}
    pA_{2,\mathcal{E}}(p) \ = \ p^2+O(p^{3/2}),
\end{eqnarray} with the lower order terms of size $p^{3/2}$, $p$, $p^{1/2}$ and $1$ having important cohomological interpretations.
\end{theorem}

It is possible to have terms of size $p^{3/2}$; see for example \cite{Mi3}.

\begin{theorem}[Birch's Theorem] For the family $\mathcal{E} : y^2 = x^3 + ax + b$ of all elliptic curves, the second moment of the Dirichlet coefficients equals
\begin{eqnarray}
    pA_{2,\mathcal{F}}(p)\ =\ \sum\limits_{a,b \bmod p} a_{\mathcal{E}}(p) \ =  \ p^3-p^2.
\end{eqnarray}
\end{theorem}

See \cite{Bi, Mi1, Mi3, Mic}.


We provide detailed calculations in \S\ref{sec:calculation} and \S\ref{sec:representative families} for some families to illustrate the techniques. See \cite{MWe} and \cite{Wu} for the comprehensive calculations for all of the one-parameter families where we were able to obtain closed form expressions. We then turn to families where we cannot obtain closed form expressions, higher rank families, and higher moments. We provide some details from representative families here, and refer the reader to \cite{MWe} for more. The following summary, taken from the appendix by Miller and Weng in \cite{ACFKKLMMWWYY}, summarizes these results.

\begin{itemize}

\item All the rank $0$ and rank $1$ families studied have data consistent with a negative bias in their second moment sums. However, for higher rank families (rank at least 4) the data suggests that there is instead a positive bias.

\item For the fourth moment, we also believe that the rank $0$ and rank $1$ families have negative biases. We see this in some families; in others we
see the presence of terms of size $p^{5/2}$ whose behavior is consistent with their averaging to zero, but its presence makes it impossible to detect the lower order terms. Interestingly, again for higher rank families (rank at least 4) the data is consistent with a positive bias.

\item The sixth moment results are similar to the fourth moment. The results are consistent with either a negative bias, or a leading term (now of
size $p^{7/2}$) averaging to zero for lower rank families. Our data also suggests that higher rank families have positive biases.

\item For the odd moments, the coefficients of the leading term vary with the primes. Our data suggests that the average value of the main term for
the $(2k+1)$\textsuperscript{st} moment is $-C_{k+1}rp^{k+1}$, where $C_n = \frac1{n+1}\ncr{2n}{n}$ is the $n$\textsuperscript{th} Catalan number.
\end{itemize}

\section{Tools for Calculating Biases}\label{sec:toolscompute}

In this section we explain why the first moment is often related to the rank, and then introduce the linear and quadratic Legendre sums, the Jacobi symbol as well as the Gauss Sum Expansion, which can be used to compute biases in elliptic curves. See more details from \cite{RoSi, BEW, BAU, Mi1}.

\begin{theorem}[Rosen-Silverman] For an elliptic surface (a one-parameter family), if Tate's conjecture holds, the first moment is related to the rank of the family over $\mathbb{Q}(T)$:

\begin{eqnarray}
    \lim_{x\to\infty}\frac{1}{x}\sum\limits_{p\leq x} \frac{A_1,\mathcal{E}(p)\log{p}}{p} = {\rm rank \mathcal{E}(\mathbb{Q}(T))}.
\end{eqnarray}
\end{theorem}

\begin{conjecture} [Tate's Conjecture for Elliptic Surfaces \cite{ST}]
Let $\mathcal{E}/\mathbb{Q}$ be an elliptic surface and $L_2(\mathcal{E},s)$ be the $L$-series attached to ${H^2}_{et}(\mathcal{E}/\mathbb{Q},\mathbb{Q}_{l})$. Then $L_2(\mathcal{E},s)$ has a meromorphic continuation to $\mathcal{C}$ and satisfies
\begin{eqnarray}
    -{\rm ord}_{s=2}L_2(\mathcal{E},s) = {\rm rank} NS(\mathcal{E}/\mathbb{Q}),
\end{eqnarray}
where $NS(\mathcal{E}/\mathbb{Q})$ is the $\mathbb{Q}$-rational part of the Neron-Severi group of $\mathcal{E}$. Further, $L_2(\mathcal{E},s)$ does not vanish on the line $Re(s)=2$.
\end{conjecture}
\noindent
Tate's conjecture is known to hold for rational surfaces: An elliptic curve $y^2=x^3+A(T)x+B(T)$ is rational if and only if one of the following is true:
\begin{enumerate}

\item $0 < \max (3\deg A, 2\deg B) < 12$,

\item $3\deg A = 2\deg B = 12 \ \text{and} \ {\rm ord}_{T=0}T^{12}\Delta(T^{-1})=0$.

\end{enumerate}

All of the one-parameter families we compute are rational surfaces; A representative case is done in \S\ref{sec:calculation}. However, for two-parameter families, we cannot use the Rosen-Silverman theorem, and the ranks are conjectural. Checking their ranks is beyond the scope of this paper, but it can be done; see \cite{WAZ} for more details. As our interest is in the biases of the second moments, we do not need to know these ranks for our purposes.

The key to our analysis in the families below are closed form expressions for linear and quadratic Legendre sums.

\begin{lemma}\label{lem:sumlegendrelinquad} Let $a, b, c$ be positive integers and $a \not\equiv 0 \bmod p$. Then
\begin{eqnarray}
    \sum_{x \bmod p}\js{ax+b} \ = \ 0, \ {\rm if} \ p \nmid a,
\end{eqnarray}
\noindent
and
\begin{eqnarray}
  \sum\limits_{x \bmod p}\js{ax^2+bx+c} \ = \
  \begin{cases}
    -\js{a},& \text{{\rm if}} \ p \nmid b^2-4ac\\
    (p-1)\js{a},              & \text{{\rm if}} \ p \mid b^2-4ac.
  \end{cases}
\end{eqnarray}
\end{lemma}

The proofs are standard and follow from elementary manipulations of the sums, exploiting changes of variables modulo $p$. For details see Appendix C \cite{Mi1} (available online). For example the first is proved by sending $x$ to $a^{-1}(x-b)$, which yields \be \sum_{x \bmod p}\js{ax+b} \ = \ \sum_{x \bmod p} \js{x}, \ee which is zero as there are as many non-zero squares as non-squares modulo $p$.

In many families we end up with terms such as $\js{-1}$. By Dirichlet's theorem for primes in arithmetic progression, to first order as $N$ tends to infinity there are the same number of primes $p \le N$ congruent to $1 \bmod 4$ as there are congruent to $3 \bmod 4$. Thus, up to lower order terms tending to zero as $N$ goes to infinity, the average is controlled by the following:
\begin{eqnarray}
\js{-1} & \ = \ &  \begin{cases} 1, \ \text{\rm if} \ p \equiv 1 \bmod 4 \\
-1, \ \text{\rm if} \ p \equiv 3 \bmod 4. \\
\end{cases}
\end{eqnarray}
See \cite{Var} for more details.

For some families, an alternative expansion for the Dirichlet coefficients is useful.

\begin{lemma}[Quadratic Formula mod $p$]\label{lem:quadformulamodp}
    For a quadratic $ax^2+bx+c \equiv 0 \mod p$, $a \not \equiv 0$, there are two distinct roots if $b^2-4ac$ equals to a non-zero square, one root if $b^2-4ac \equiv 0$, and zero roots if $b^2-4ac$ is not a square.
\end{lemma}




\section{Representative one-parameter Family}\label{sec:calculation}
\begin{lemma}The first moment of the family $y^2=4x^3+ax^2+bx+c+dt$ is 0.
\end{lemma}

\begin{proof}
For all $p>4d$, send $t$ to $4d^{-1}t$: Thus
\be \sum_{t(p)} \js{dt}=\sum_{t(p)} \js{4t} .\ee
Therefore,
\bea A_{1,\varepsilon(p)}&\ = \ & -\sum_{t(p)} \sum_{x(p)}  \js{4x^3+ax^2+bx+4t+c}\nonumber\\
&\ = \ & -\sum_{x(p)} \sum_{t(p)} \js{4t+4x^3+ax^2+bx+c}.\eea
As $p\not|4$ when $p\ne 2$, the t-sum vanishes by linear sum theorem, and
\be A_{1,\varepsilon(p)} \ = \ 0.  \ee
By the Rosen-Silverman Theorem, this is a rank 0 family. \end{proof}

\begin{lemma} The second moment of the family $y^2=4x^3+ax^2+bx+c+dt$ is
\begin{equation}
\twocase{A_{2,E}(p)\ =\ }{p^2-p-p\cdot \js{-48}-p\cdot \js{a^2-12b}}{if $a^2-12b\ne 0$}{p^2-p+p(p-1)\js{-48}}{otherwise.}
\end{equation}
\end{lemma}

\begin{proof} We have
\bea A_{2,E}(p)&\ = \ & \sum_{t(p)} \sum_{x(p)} \sum_{y(p)} \js{4x^3+ax^2+bx+4t+c} \js{4y^3+ay^2+by+4t+c} \nonumber\\ m(x)&\ = \ &4x^3+ax^2+bx+c \nonumber\\
n(y)&\ = \ &4y^3+ay^2+by+c \nonumber\\
A_{2,E}(p)&\ = \ & \sum_{t(p)} \sum_{x(p)} \sum_{y(p)}  \js{16t^2+4(m+n)t+mn}.
\eea
The discriminant of $16t^2+4(m+n)t+mn$ is
\bea
\Delta_t(x,y)&\ = \ &16(m+n)^2-64mn \nonumber\\
&\ = \ &16(m-n)^2 \nonumber\\
\delta^2&\ = \ &\Delta_t(x,y) \nonumber\\
\delta&\ = \ &4(m-n)\nonumber\\
&\ = \ &4(4x^3+ax^2+bx+c-4y^3-ay^2-by-c)\nonumber\\
&\ = \ &4(x-y)(4x^2+4xy+4y^2+ax+ay+b).
\eea

If $p|\delta$, then $p|x-y$ or $p|4x^2+4xy+4y^2+ax+ay+b$. $x, y$ range from $0$ to $p-1$, so $p|x-y$ exactly $p$ times. By the quadratic formula mod $p$, $4x^2+4xy+4y^2+ax+ay+b\equiv 4y^2+(4x+a)y+4x^2+ax+b\equiv 0 \pmod p$ when $y=\frac{-4x-a\pm \sqrt{\Delta_y}}{8}$  ($\Delta_y$ is the discriminant of the polynomial  $4y^2+(4x+a)y+4x^2+ax+b$ in terms of y):
\bea
\Delta_y&\ =\ &(4x+a)^2-4\cdot 4(4x^2+ax+b)\nonumber\\
&\ = \ &-48x^2-8ax+a^2-16b.
\eea
If $\Delta_y$ is a non-zero square mod $p$, there are two solutions. If $\Delta_y$ is 0 mod $p$, there is one solution. If $\Delta_y$ is not a square mod $p$, there is no solution.
\\
The number of pairs of $x,y$ such that $p|4x^2+4xy+4y^2+ax+ay+b$ is
\be
\sum_{x(p)} 1+\js{-48x^2-8ax+a^2-16b}=p+\sum_{x(p)} \js{-48x^2-8ax+a^2-16b}.
\ee
The discriminant of $-48x^2-8ax+a^2-16b$ is
\bea
\Delta_x&\ = \ &(8a)^2-4\cdot (-48)(a^2-16b)\nonumber\\
&\ = \ &256a^2-3072b\nonumber\\
&\ = \ &256(a^2-12b).
\eea
\ \\

We break into cases, depending on the value of the discriminant.

\noindent \emph{Case 1: $a^2-12b\ne 0$}:
By the Quadratic Legendre Sum Theorem, if $p\not |256(a^2-12b)$,
\be \sum_{x(p)} \js{-48x^2-8ax+a^2-16b}=-\js{-48}. \ee
The number of pairs of $x,y$ such that $p|4x^2+4xy+4y^2+ax+ay+b$ is
\be
p+\sum_{x(p)} \js{-48x^2-8ax+a^2-16b}=p-\js{-48}.
\ee

The cases that we double count $x=y$ and $p|4x^2+4xy+4y^2+ax+ay+b$ is
\be
4x^2+4xy+4y^2+ax+ay+b\equiv 12y^2+2ay+b\equiv0 \pmod p. \label{acdde1}
\ee
The discriminant of $12y^2+2ay+b$ is
\be
\Delta_y=(2a)^2-4\cdot 12b=2^2\cdot (a^2-12b).
\ee

By the quadratic formula mod $p$, the number of solutions is computable, depending on $a^2-12b$.

\ \\
The number of solutions to (\ref{acdde1}) is
\be
1+\js{2^2\cdot (a^2-12b)}=1+\js{a^2-12b}.
\ee
Therefore, the total number of times that $p|(x-y)(4x^2+4xy+4y^2+ax+ay+b)$ is the number of times $p|x-y$ plus the number of times $p|4x^2+4xy+4y^2+ax+ay+b$ minus the cases that we double count.
\ \\
The number of times that $p|(x-y)(4x^2+4xy+4y^2+ax+ay+b)$ is
\be
p+p-\js{-48}-1-\js{a^2-12b}=2p-1-\js{-48}-\js{a^2-12b}.
\ee
The number of times that $p\not |(x-y)(4x^2+4xy+4y^2+x+y-4)$ is
\be
p^2-\left(2p-1-\js{-48}-\js{a^2-12b}\right)=p^2-2p+1+\js{-48}+\js{a^2-12b}.
\ee
By Quadratic Legendre Sum Theorem,
\bea
A_{2,E}(p)&\ = \ &(p-1)\left[2p-1-\js{-48}-\js{a^2-12b}\right]-\left[p^2-2p+1+\js{-48}+\js{a^2-12b}\right]\nonumber\\
&\ = \ &p^2-p-p\cdot \js{-48}-p\cdot \js{a^2-12b}.
\eea

\ \\

\noindent \emph{Case 2: $a^2-12b=0$}: By the Quadratic Legendre Sum Theorem, since $p|0$, we have
\be
\sum_{x(p)} \js{-48x^2-8ax+a^2-16b}=(p-1)\js{-48}.
\ee
The number of pairs of $x,y$ such that $p|4x^2+4xy+4y^2+ax+ay+b$ is
\be
p+\sum_{x(p)} \js{-48x^2-8ax+a^2-16b}=p+(p-1)\js{-48}.
\ee

The cases that we double count $x=y$ and $p|4x^2+4xy+4y^2+ax+ay+b$ is
\be
4x^2+4xy+4y^2+ax+ay+b\equiv 12y^2+2ay+b\equiv0 \pmod p. \label{acdde}
\ee
The discriminant of $12y^2+2ay+b$ is
\be
\Delta_y=(2a)^2-4\cdot 12b=2^2\cdot (a^2-12b).
\ee

By the quadratic formula mod $p$, the number of solutions is computable, depending on $a^2-12b$.

The number of solutions to (\ref{acdde}) is
\be
1+\js{2^2\cdot (a^2-12b)}=1.
\ee
Therefore, the total number of times that $p|(x-y)(4x^2+4xy+4y^2+ax+ay+b)$ is the number of times $p|x-y$ plus the number of times $p|4x^2+4xy+4y^2+ax+ay+b$ minus the cases that we double count.
\ \\
The number of times that $p|(x-y)(4x^2+4xy+4y^2+ax+ay+b)$ is
\be
p+p+(p-1)\js{-48}-1=2p-1+(p-1)\js{-48}.
\ee
The number of times that $p\not |(x-y)(4x^2+4xy+4y^2+x+y-4)$ is
\be
p^2-\left(2p-1+(p-1)\js{-48}\right)=p^2-2p+1-(p-1)\js{-48}.
\ee
    
\bea
A_{2,E}(p)&\ = \ &(p-1)\left[2p-1+(p-1)\js{-48}\right]-\left[p^2-2p+1-(p-1)\js{-48}\right]\nonumber\\
&\ = \ &p^2-p+p(p-1)\js{-48}.
\eea

Thus we have shown \begin{equation}
A_{2,E}(p) =
\Bigg\{ {\  p^2-p-p\cdot \js{-48}-p\cdot \js{a^2-12b} \ \
\mbox{if} \  a^2-12b\ne 0 \atop p^2-p+p(p-1)\js{-48} \ \mbox{otherwise.}}
\end{equation}
\end{proof}

\section{Numerical computations for biases in second moments}

Unfortunately, for most families we cannot obtain a closed form for the first or second moments. We thus report on some numerical investigation of one-parameter families; we explore several different ranks to see if that has any impact. For some families, we are able to conjecture a formula for the second moment by separating primes into different congruence classes, which suggests that there is often a closed-form polynomial expression. We are then able to prove the results mathematically in some cases. The following table summarizes the numerical results for the second moments' expansions of the families we studied; see Figure \ref{fig:dimensionlinesquareD}.

\begin{figure}[ht]
\begin{center}
\scalebox{1.2}{\includegraphics{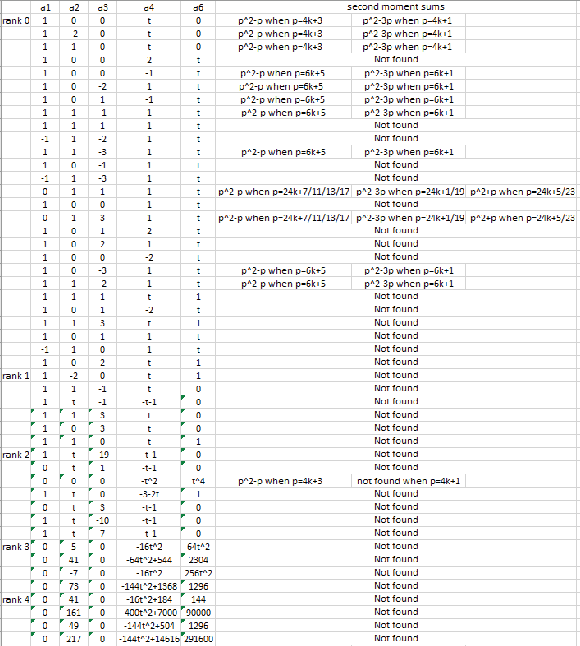}}
\caption{\label{fig:dimensionlinesquareD} Systematic investigation for second moments sums.}
\end{center}
\end{figure}
\ \\

We summarize the first and second moments for some families where we are able to prove closed form expressions for the first two moments. The arguments are representative of the ones needed for all the families. The proofs are similar to the ones in Section \ref{sec:calculation}

\ \\
\noindent \textbf{Family: $y^2=4x^3+ax^2+bx+c+dt$:}
\begin{itemize}
\item First moment: $A_{1,\varepsilon(p)}=0$.
\item Second moment:
\begin{equation}
\twocase{A_{2,\varepsilon(p)} \ = \ }{p^2-p-p\cdot \js{-48}-p\cdot \js{a^2-12b}}{if $a^2-12b\neq 0$}{p^2-p+p(p-1)\js{-48}}{otherwise.}
\end{equation}
\end{itemize}

\ \\
\noindent \textbf{Family: $y^2=4x^3+(4m+1)x^2+n\cdot tx$:}
\begin{itemize}
\item First moment: $A_{1,\varepsilon(p)}=0$.
\item Second moment: \begin{equation} A_{2,\varepsilon(p)}\ =\
\Bigg\{ {\ p^2-3p \ \
\mbox{if} \  p=4k+1 \atop p^2-p \ \mbox{if} \ p=4k+3.}
\end{equation}
\end{itemize}

\ \\
\noindent \textbf{Family: $y^2=x^3-t^2x+t^4$:}
\begin{itemize}
\item First moment: $A_{1,\varepsilon(p)} \ = \ -2p$.
\item Second moment:
\begin{equation}
A_{2,\varepsilon(p)}=p^2-p-p\cdot \js{-3}-p\cdot \js{12}-\sum_{x(p)} \sum_{y(p)} \js{x^3-x} \js{y^3-y}.
\end{equation}
\end{itemize}
\ \\

For families that we are not able to find closed-form expressions, we calculated the average bias of the second moment sums for the first $1000$ primes; see Figure \ref{fig:severalmoments}.

\begin{figure}[ht]
\begin{center}
\includegraphics[width=0.9\linewidth, height=260px]{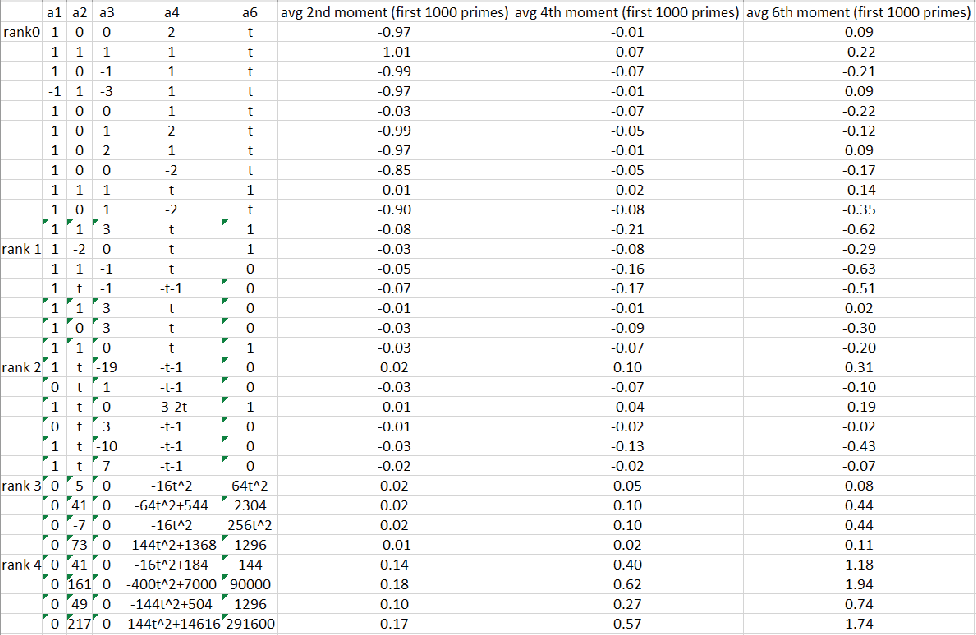}
\caption{Numerical data for the average biases of 2nd, 4th and 6th moments sums.}
\label{fig:severalmoments}
\end{center}
\end{figure}

By Michel's theorem, we know that the main term of the sum is $p^2$, and lower order terms have size $p^{3/2}, p, p^{1/2}$ or $1$. From the data we have, we can tell if it is likely that the second moment has a $p^{3/2}$ term. If the value of $\frac{\text{second\ moment}-p^2}{p}$ converges or stays bounded as the prime grows, then it is likely that the largest lower order term of the second moment sum is $p$, as if there were a $p^{3/2}$ term we would have fluctuations of size $p^{1/2}$.

By subtracting the main term $p^2$ from the sum and then dividing by the largest lower term ($p^{3/2}$ or $p$), we calculated the average bias; see Figure \ref{bias_2nd}.

\begin{figure}[ht]
\begin{center}
\scalebox{1}{\includegraphics{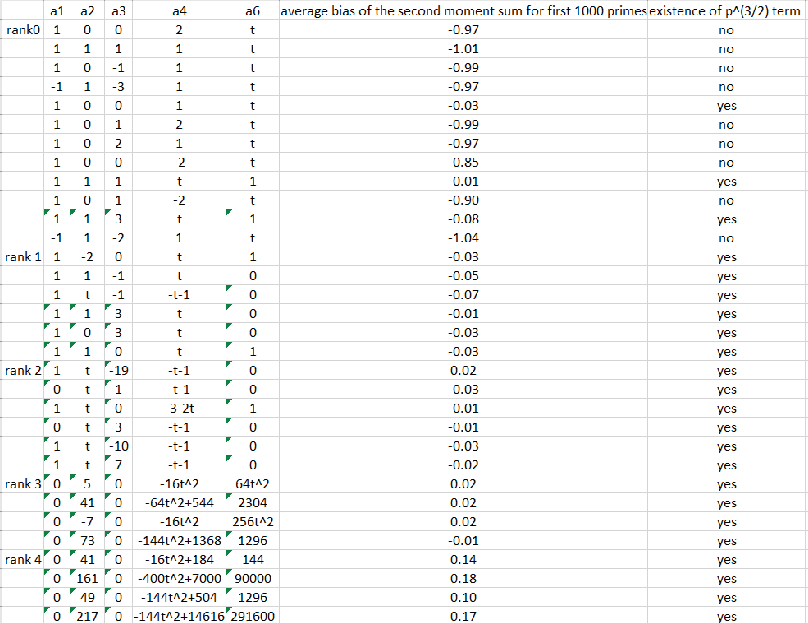}}
 \caption{\label{bias_2nd} Numerical data for the average biases of second moments sums.}
\end{center}
\end{figure}

The data shows that all the families where we do not believe there is a $p^{3/2}$ term clearly have negative biases (around $-1$). When the $p^{3/2}$ exists, the bias unfortunately becomes impossible to see. The reason is that the $p^{3/2}$ term drowns it out; we now have to divide by $p^{3/2}$. If that term averages to zero, then the term of size $p$, once we divide by $p^{3/2}$, is of size $1/p^{1/2}$.

Let's investigate further the consequence of having a term of size $p^{3/2}$. We divide the difference of the observed second moment minus $p^2$ (the expected value) by $p^{3/2}$. We now have signed summands of size 1. By the Philosophy of Square-Root Cancelation, if we sum $N$ such signed terms we expect a sum of size $\sqrt{N}$. As we are computing the average of these second moments, we divide by $N$ and have an expected value of order $1/\sqrt{N}$. In other words, if the $p^{3/2}$ term is present and averages to zero, we expect sums over ranges of primes to be about $1/\sqrt{N}$. If $N=1000$ this means we expect sums on the order of .0316. Looking at the data in Figure \ref{bias_2nd}, what we see is consistent with this analysis. Thus, while we cannot determine if the first lower order term that does not average to zero has a negative bias, we can at least show that the data is consistent with the $p^{3/2}$ term averaging to zero for lower rank families.

The table suggests a lot more. From the data, we can see that all the rank $0$ and rank $1$ families have negative biases. However, all four rank $4$ families have shown positive biases from the first 1000 primes. Thus, we look further to see if it is likely the result of fluctuations, or if perhaps it is evidence against the bias conjecture.

\ \\
We now list the results for a few representative families.

\ \\
We divide the $1000$ primes into $20$ groups of $50$ for further analysis. If the $p^{3/2}$ term averages to zero, we would expect each of these groups to be positive and negative equally likely, and we can compare counts. We now expect each group to be on the order of $1/\sqrt{50} \approx .14$. Thus we shouldn't be surprised if it is a few times .14 (positive or negative); remember we do not know the constant factor in the $p^{3/2}$ term and are just doing estimates.

\ \\
For the rank $2$ family $a_1=1$, $a_2=t$, $a_3=-19$, $a_4=-t-1$, $a_6=0$, $12$ of the $20$ groups of primes have shown positive biases. Figure \ref{1t(-19)(-t-1)0} is a histogram plot of the distribution of the average biases among the 20 groups.

\begin{figure}[ht]
\begin{center}
\scalebox{1}{\includegraphics{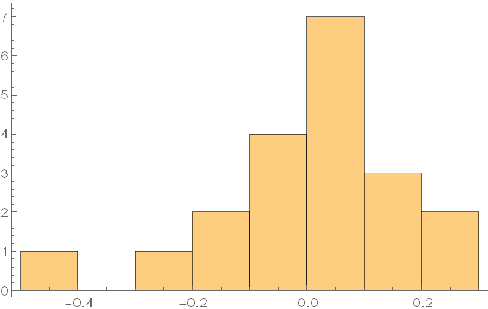}}
\caption{\label{1t(-19)(-t-1)0} Distribution of average biases in the first 1000 primes for family $a_1=1$, $a_2=t$, $a_3=-19$, $a_4=-t-1$, $a_6=0$.}
\end{center}
\end{figure}

\ \\
We now further analyze our data by dividing the 1000 primes into 100 groups of 10 for rank $6$ family $a_1=0$, $a_2=2(16660111104 t) + 811365140824616222208$, $a_3=0$, \newline $a_4=[2(-1603174809600)t-26497490347321493520384](t^2+2t-8916100448256000000+1)$,\newline $a_6=[2(2149908480000)t+343107594345448813363200](t^2+2t-8916100448256000000+1)^2$. Our data suggests that there may be a positive bias in this family; in other words, the bias conjecture may fail if the family rank is sufficiently large. The average bias of second moments sums for the first $1000$ primes is $0.246759$. Figure \ref{811365140824616222208_2nd} is a histogram plot of the distribution of the average biases among the $100$ groups of $10$ primes. Note $78$ of the $100$ groups of primes have positive biases, which suggests that it is likely that the second moment of this family has a positive $p^{3/2}$ term.

\begin{figure}[ht]
\begin{center}
\scalebox{0.9}{\includegraphics{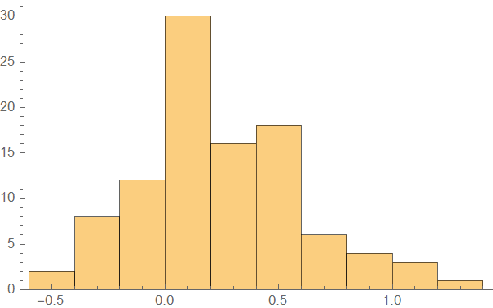}}
\caption{\label{811365140824616222208_2nd} Distribution of average biases in the first 1000 primes for a rank 6 family.}
\end{center}
\end{figure}

From all the data we've collected, the rank $0$ and rank $1$ families have negative biases more frequently, but we are working with small data sets and must be careful in how much weight we assign such results. Our data for the rank $4$ and rank $6$ families has shown that it is likely that higher rank families (${\rm rank}(E(\Q))\ge4$) have positive biases.


\section{Biases in higher moments}
We now explore, for the first time, the higher moments of the Dirichlet coefficients of the elliptic curve $L$-functions to see if biases we found in the first and second moments persist. Unfortunately existing techniques on analyzing the second moment sums do not apply to the higher moments, even if we choose nice families. If we switch orders of the moments' sums and sum over $t$, we are going to get a cubic or higher degrees polynomials. Therefore, we can only try to predict or observe the biases through numerical evidence. We calculated the 4th and 6th moment sums for the first $1000$ primes. From Section \ref{4th_moment_form}, we know that the main term of the fourth moment sum is $2p^3$, and the largest possible lower order terms have size $p^{5/2}$. From Section \ref{6th_moment_form}, we know that the main term of the sixth moment sum is $5p^4$, and the largest possible lower order terms have size $p^{7/2}$. From the data we have gathered, all the 4th moments of these families have $p^{5/2}$ terms, and all the 6th moments have $p^{7/2}$ terms. By subtracting the main term $2p^3$ from the fourth moment sum and then dividing by the size of the largest lower term $p^{5/2}$, we calculated the average bias for the fourth moment of the first $1000$ primes. Similarly, we subtracted $5p^4$ from the sixth moment sum and then divided by $p^{7/2}$ to calculate the average bias for the sixth moment of the first $1000$ primes; See Figure \ref{4_6moments}.

\begin{figure}[ht]
\begin{center}
\scalebox{1}{\includegraphics{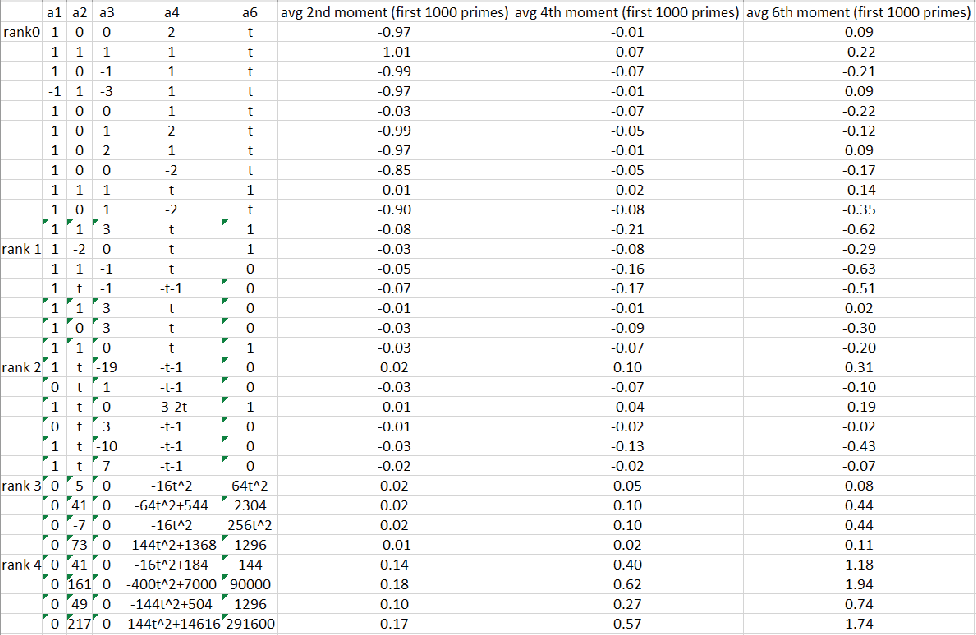}}
\caption{\label{4_6moments}Numerical data for the average biases of 2nd, 4th and 6th moments sums.}
\end{center}
\end{figure}

\subsection{Biases in fourth moment sums}
From the data, we can see that all the biases for lower rank families in the fourth moment are relatively small (smaller than $0.2$), which indicates that the $p^{5/2}$ term likely averages to $0$. By the Philosophy of Square-Root Cancellation, we expect the order of the size of fluctuation to be around $\sqrt{1000}/1000 \approx 0.03$. Therefore, if the bias is between $-0.2$ and $0.2$, we would expect $p^2$ to be the largest lower order term.

Note that for $30$ out of $31$ families, the bias in fourth moments appear to be similar to the bias in second moments (families that have negative bias in second moments also seem to have negative bias in fourth moments, and vice versa), though much smaller magnitudes likely due to the presence of a $p^{5/2}$ term that is averaging to zero. We now explore a few representative families whose 4-th moment biases have different scales in magnitudes.

For the rank $1$ family $a_1=1$, $a_2=t$, $a_3=-1$, $a_4=-t-1$, $a_6=0$, we analyze our data by dividing the 1000 primes into 100 groups of 10. As shown in Figure \ref{1t-1-t-10_4th_10primes}, $63$ of the $100$ groups of primes have shown negative biases. The probability of having 17 or more negatives than positives (or 17 or more positives than negatives) in 100 tosses of a fair coin (so heads is positive and tails is negative) is about 1.2\%. While unlikely, this is not exceptionally unlikely.

\begin{figure}[ht]
\begin{center}
\scalebox{.95}{\includegraphics{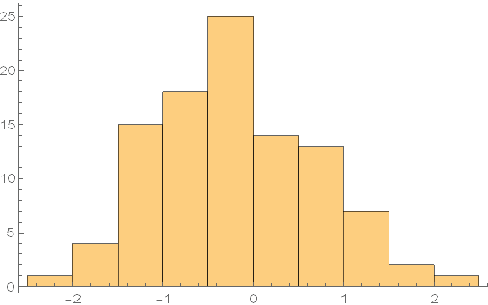}}
\caption{\label{1t-1-t-10_4th_10primes} Distribution of average biases in the first 1000 primes for family $a_1=1$, $a_2=t$, $a_3=-1$, $a_4=-t-1$, $a_6=0$.}
\end{center}
\end{figure}

For the rank $3$ family $a_1=0$, $a_2=5$, $a_3=0$, $a_4=-16t^2$, $a_6=64t^2$, $11$ of the $20$ groups of primes have shown negative biases. Figure \ref{050-16t^264t^2_4th} is a histogram plot of the distribution of the average biases among the 20 groups.

\begin{figure}[ht]
\begin{center}
\scalebox{0.75}{\includegraphics{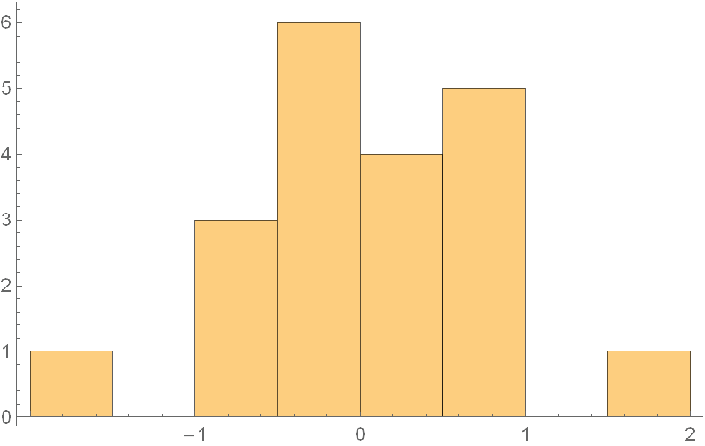}}
\caption{\label{050-16t^264t^2_4th} Distribution of average biases in the first 1000 primes for family $a_1=0$, $a_2=5$, $a_3=0$, $a_4=-16t^2$, $a_6=64t^2$.}
\end{center}
\end{figure}

Despite the fluctuations, all the rank $0$ and rank $1$ families seem to have negative biases more frequently in the first $1000$ primes, which suggests that it is possible that negative bias exists in the fourth moments of all rank $0$ and rank $1$ families.
Similar to the second moment sums, the fourth moment sums of families with larger rank appear to have positive biases for the first 1000 primes, but this might due to the fluctuations of the $p^{5/2}$ term as we are working with small data set.

To further examine the biases in families with larger ranks, we investigate the rank 6 family $a_1=0$, $a_2=2(16660111104 t) + 811365140824616222208$, $a_3=0$, $a_4=[2(-1603174809600)t-26497490347321493520384](t^2+2t-8916100448256000000+1)$, $a_6=[2(2149908480000)t+343107594345448813363200](t^2+2t-8916100448256000000+1)^2$.
Our data suggests that there is a positive bias in this family. The average bias of the fourth moments sums for the first $1000$ primes is $0.753285$. Figure \ref{811365140824616222208_4th} is a histogram plot of the distribution of the average biases among the $100$ groups of $10$ primes. $75$ of the $100$ groups of primes have positive biases, which suggests that it is likely that the fourth moment of this family has a positive $p^{5/2}$ term.

\begin{figure}[ht]
\begin{center}
\scalebox{1}{\includegraphics{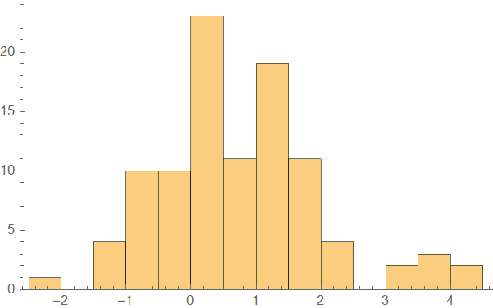}}
\caption{\label{811365140824616222208_4th} Distribution of average biases in the first 1000 primes for a rank 6 family.}
\end{center}
\end{figure}

To sum up, the fourth moments of all lower rank families have first lower order terms that average to $0$ or negative. Similar to the second moment, our data for the rank $4$ and rank $6$ families suggests that higher rank families might have positive biases.

\subsection{Biases in sixth moment sums}
We now explore the 6th moment biases for these families.

For the rank $0$ family $a_1=1$, $a_2=1$, $a_3=1$, $a_4=1$, $a_6=t$, $13$ of the $20$ groups of primes have shown negative biases. Figure \ref{1111t_6th} is a histogram plot of the distribution of the average biases among the 20 groups.

\begin{figure}[ht]
\begin{center}
\scalebox{0.75}{\includegraphics{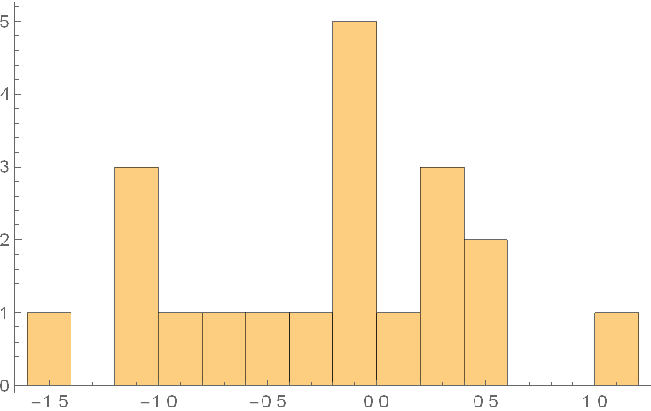}}
\caption{\label{1111t_6th} Distribution of average biases in the first 1000 primes for family $a_1=1$, $a_2=1$, $a_3=1$, $a_4=1$, $a_6=t$.}
\end{center}
\end{figure}

For the rank $1$ family $a_1=1$, $a_2=t$, $a_3=-1$, $a_4=-t-1$, $a_6=0$, $14$ of the $20$ groups of primes have shown negative biases. Figure \ref{1t-1-t-10_6th} is a histogram plot of the distribution of the average biases among the 20 groups.

\begin{figure}[ht]
\begin{center}
\scalebox{0.75}{\includegraphics{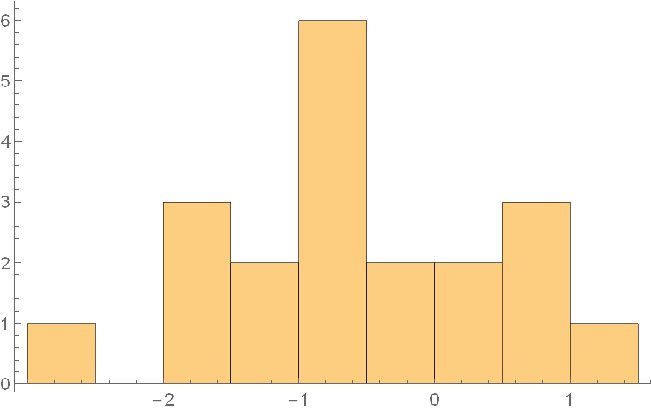}}
\caption{\label{1t-1-t-10_6th} Distribution of average biases in the first 1000 primes for family $a_1=1$, $a_2=t$, $a_3=-1$, $a_4=-t-1$, $a_6=0$.}
\end{center}
\end{figure}

For the rank $3$ family $a_1=0$, $a_2=5$, $a_3=0$, $a_4=-16t^2$, $a_6=64t^2$, $10$ of the $20$ groups of primes have shown negative biases. Figure \ref{050-16t^264t^2B} is a histogram plot of the distribution of the average biases among the 20 groups.

\begin{figure}[ht]
\begin{center}
\scalebox{0.75}{\includegraphics{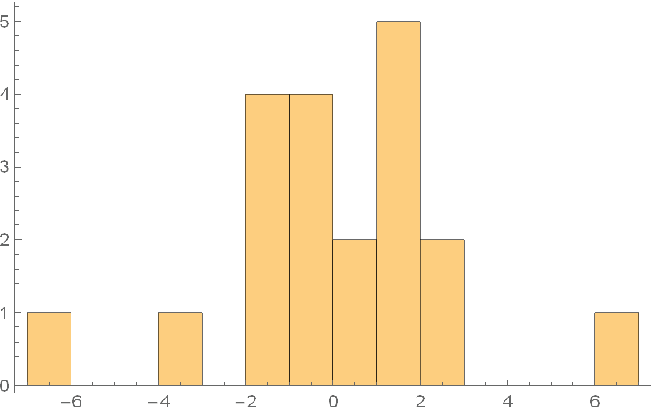}}
\caption{\label{050-16t^264t^2B} Distribution of average biases in the first 1000 primes for family $a_1=0$, $a_2=5$, $a_3=0$, $a_4=-16t^2$, $a_6=64t^2$.}
\end{center}
\end{figure}

As shown from the data, for most families, it is inconclusive whether the 6th moments have negative biases. Our data strongly suggests that the $p^{7/2}$ term averages to $0$ for most families , and the $p^3$ term is drowned out by the fluctuations of $p^{7/2}$ term.

For the rank 6 family $a_1=0$, $a_2=2(16660111104 t) + 811365140824616222208$, $a_3=0$, $a_4=[2(-1603174809600)t-26497490347321493520384](t^2+2t-8916100448256000000+1)$, $a_6=[2(2149908480000)t+343107594345448813363200](t^2+2t-8916100448256000000+1)^2$, the average bias of the sixth moments sums for the first $1000$ primes is $2.26$. Figure \ref{811365140824616222208_6th} is a histogram plot of the distribution of the average biases among the $100$ groups of $10$ primes. $69$ of the $100$ groups of primes have positive biases, which suggests that it is likely that the sixth moment of this family has a positive $p^{7/2}$ term.

\begin{figure}[ht]
\begin{center}
\scalebox{0.9}{\includegraphics{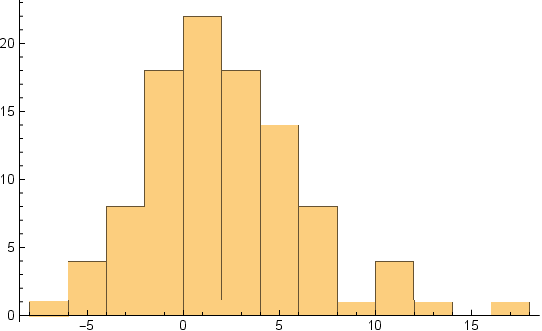}}
\caption{\label{811365140824616222208_6th} Distribution of average biases in the first 1000 primes for a rank 6 family.}
\end{center}
\end{figure}

To sum up, while we are not able to tell if the negative bias exists in the higher even moments, the data is at least consistent with the first lower order term averaging to zero or negative for families with lower ranks. Thus our numerics support a weaker form of the bias conjecture: the first lower order term does not have a positive bias for lower rank families. For families with rank at least $4$, the negative bias conjecture may not hold.

\subsection{Biases in the Third, Fifth, and Seventh Moments}
We now explore the third, fifth, and seventh moments of the Dirichlet coefficients of  elliptic curve $L$-functions. By the Philosophy of Square-Root Cancellation, $p$ times the third moment, $p$ times the fifth moment, and $p$ times the seventh moment  should have size $p^2$, $p^3$, $p^4$ respectively (we are multiplying by $p$ to remove the $1/p$ averaging). For example, the third moment is a sum of $p$ terms, each of size $\sqrt{p}^3$. Thus as these are signed quantities, we expect the size to be on the order of $\sqrt{p} \cdot p^{3/2}$.

We believe that there are bounded functions $c_{3,\mathcal{E}}(p)$, $c_{5,\mathcal{E}}(p)$, and $c_{7,\mathcal{E}}(p)$ such that
\be A_{3,\mathcal{E}}(p) =c_{3,\mathcal{E}}(p) p^2+O(p^{3/2}), \ \ \ A_{5,\mathcal{E}}(p) =c_{5,\mathcal{E}}(p) p^3+O(p^{5/2}), \ \ \ A_{7,\mathcal{E}}(p) =c_{7,\mathcal{E}}(p) p^4+O(p^{7/2});\ee our data supports these conjectures. Unlike the second, fourth, and sixth moments, the coefficient of the leading term can vary with the prime in the third, fifth, and seventh moments. We calculated the average values of $c_{3,\mathcal{E}}(p)$, $c_{5,\mathcal{E}}(p)$, and $c_{7,\mathcal{E}}(p)$ for each elliptic curve family by dividing the size of the main term ($p^2$ for third moment, $p^3$ for fifth moment, and $p^4$ for the seventh moment); see Figure \ref{3_5moments}.

\begin{figure}[H]
\begin{center}
\scalebox{1}{\includegraphics{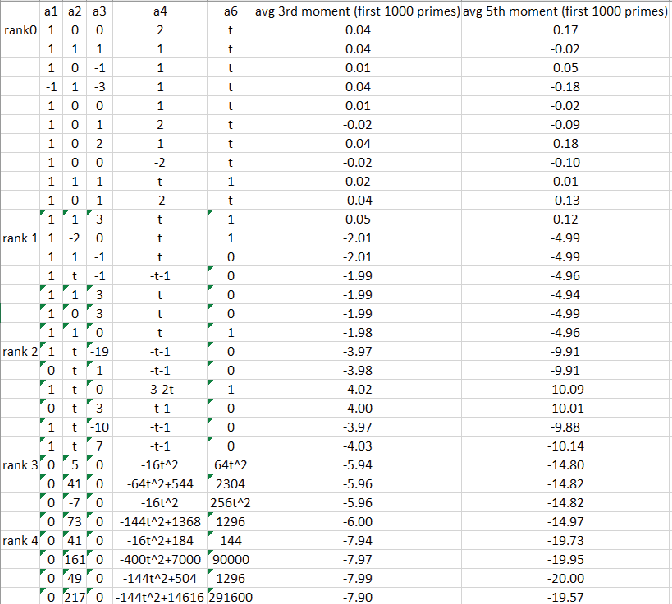}}
\caption{\label{3_5moments} Numerical data for the average constant for the main term of 3rd and 5th moments sums.}
\end{center}
\end{figure}

Our data suggests an interesting relationship between the average constant value for the main term and the rank of elliptic families for these odd moments.

\begin{conjecture} Consider a one-parameter family of elliptic curves of rank $r$.
The average value of the main term of the 3rd moment is $-2rp^2$.
\end{conjecture}

\begin{conjecture} Consider a one-parameter family of elliptic curves of rank $r$.
The average value of the main term of the 5th moment is $-5rp^3$.
\end{conjecture}

\begin{figure}[ht]
\begin{center}
\scalebox{1}{\includegraphics{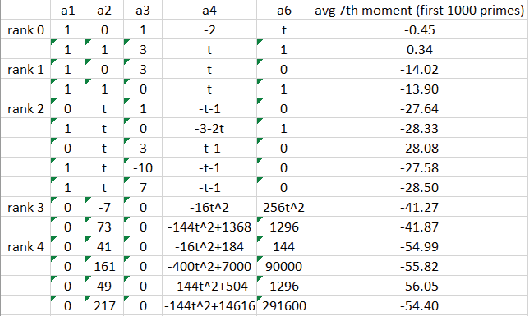}}
\caption{\label{7moments} Numerical data for the average constant for the main term of 7th moments sums.}
\end{center}
\end{figure}

\begin{conjecture} Consider a one-parameter family of elliptic curves of rank $r$.
The average value of the main term of the 7th moment is $-14rp^4$.
\end{conjecture}

\begin{conjecture} Consider a one-parameter family of elliptic curves of rank $r$.
Let $C_n$ be the $n$th Catalan number, $\frac1{n+1}\ncr{2n}{n}$. For $k \in \mathbf{Z}^+$, the average value of the main term of the $2k+1$ th moment is $-C_{k+1}rp^{k+1}$.
\end{conjecture}

We can try to analyze the third, fifth and seventh moments the same way as we did the fourth and sixth. In doing so, we would obtain expansions that do have terms related to the first moment (and hence by the Rosen-Silverman theorem the rank of the group of rational solutions); unfortunately there are other terms that arise now, due to the odd degree, that are not present in the even moments and which we cannot control as easily. We thus leave a further study of these odd moments as a future project.



\section{Conclusion and Future Work}
Natural future questions are to continue investigating the second moment bias conjecture in more and more families, theoretically if possible, numerically otherwise. Since the bias in the second moments doesn't imply biases in higher moments, we can also explore whether there is a corresponding negative bias conjecture for the higher even moments. As these will involve quartic or higher in $t$ Legendre sums, it is unlikely that we will be able to  compute these in closed form, and thus will have to resort to analyzing data, or a new approach through algebraic geometry and cohomology theory (Michel proved that the lower order terms are related to cohomological quantities associated to the elliptic curve).

Any numerical exploration will unfortunately be quite difficult in general, as there is often a term of size $p^{3/2}$ which we believe averages to zero for some families, but as it is $\sqrt{p}$ larger than the next lower order term, it completely drowns out that term and makes it hard to see the bias.

For the odd moments, our numerical explorations suggest that the bias in the first moment, which is responsible for the rank of the elliptic curve over $\Q(T)$, persists. A natural future project is to try to extend Michel's work to prove our conjectured main term formulas for the odd moments.

Another area we want to focus on in the future is getting to know the two-parameter families better. What are the implications of the negative bias of the two-parameter families? How do they behave differently from one-parameter families or other families and why?

\section{Acknowledgement}
We thank the authors of \cite{ACFKKLMMWWYY} for comments on related problems. We also thank Jiefei Wu for helpful conversations, and allowing us to expand upon her introduction from \cite{Wu} for this paper, and to include some of her results in the first appendix.

\appendix
\section{Representative one-parameter and two-parameter families (By Steven J. Miller and Jiefei Wu)}\label{sec:representative families}
\subsection{One-Parameter Family}

\begin{lemma}
The first moment of the one-parameter family $y^2=x^3-tx^2+(x-1)t^2$ is $0$ on average, and the family's rank is $0$. More explicitly, for primes greater than 3 we have $-p$ times the first moment is $2p$ if $p \equiv 1 \bmod 12$, it is $-2p$ if $p \equiv 7 \bmod 12$, and 0 otherwise, thus \be \twocase{-pA_{1,\mathcal{F}}(p) \ = \ }{2p}{if $p \equiv 1 \bmod 12$}{-2p}{if $p \equiv 7 \bmod 12$} \ee and zero for all other primes greater than 3.
\end{lemma}

\begin{proof} For $p > 3$ we compute the first moment:
\begin{eqnarray}
    -pA_{1,\mathcal{F}}(p) & \ = \ & -\sum_{t(p)}a_{t(p)} = \sum_{t(p)}\sum_{x(p)}\js{x^3-tx^2+(x-1)t^2} \nonumber\\
    & = & \sum_{t(p)}\sum_{x(p)}\js{x^3-tx^2+xt^2-t^2} \nonumber\\
    & = & \sum_{t=1}^{p-1}\sum_{x(p)}\js{t^3x^3-t^3x^2+t^3x-t^2}\nonumber\\
    & = & \sum_{x(p)}\sum_{t=1}^{p-1}\js{t^2}\js{tx^3-tx^2+tx-1}\nonumber\\
    & = & \sum_{x(p)}\sum_{t=0}^{p-1}\js{t(x^3-x^2+x)-1}-\sum_{x(p)}\js{-1}\nonumber\\
    & \ = \ & \sum_{t(p)}\sum_{x=0}\js{-1}+\sum_{t(p)}\sum_{x(p);x \not= 0}\js{t(x^3-x^2+x)-1}-\sum_{x(p)}\js{-1} \nonumber\\
    & \ = \ & \sum_{t(p)}\sum_{x(p);x \not= 0}\js{t(x^3-x^2+x)-1},\end{eqnarray}
where the first and third sums cancel in the second to last line above. For the remaining sum, if $x^3-x^2+x = x(x^2-x+1)$ is not zero modulo $p$, we can send $t$ to $(x^3-x^2+x)^{-1}t$ and thus the sum is zero. Thus the only possible contribution from this term is when $x^2 - x + 1 \equiv 0 \bmod p$ (we have already removed the $x=0$ term, so we can drop that factor). By the Quadratic Formula mod 2 (Lemma \ref{lem:quadformulamodp}) the roots are $(1 \pm \sqrt{-3})/2$. A simple calculation using quadratic reciprocity shows that if $p \equiv 1 \bmod 3$ then $-3$ is a square, and thus the two roots exist, while if $p \equiv 2 \bmod 3$ it is not a square and there are no roots. It is worth to note that all three roots are distinct when $p$ exceeds 3. In particular, $1 \pm \sqrt{3}$ cannot be zero, and these two roots are not equal to each other as that would mean $\sqrt{3}$ is zero.

So, if $p \equiv 2 \bmod 3$ the sum is zero, while if $p \equiv 1 \bmod 3$ the sum is $2p \js{-1}$. If $p \equiv 1 \bmod 4$ then $\js{-1} = 1$ while if $p \equiv 3 \bmod 4$ then $\js{-1} = -1$. By Dirichlet's Theorem for Primes in Arithmetic Progressions, we see that to first order, half the time the sum is not zero it is $2p$ while the other half of the time it is $-2p$. Thus on average it is zero, and since it is a rational surface, by the Rosen-Silverman theorem the family's rank is $0$ (we have a positive contribution when $p \equiv 1 \bmod 12$ balanced by an equal negative contribution when $p \equiv 7 \bmod 12$; the other primes greater than three contribute zero).

\end{proof}

\begin{lemma}
The second moment of the one-parameter family $y^2=x^3-tx^2+(x-1)t^2$ has a negative bias, which supports the bias conjecture, explicitly
$ p^2-2p\delta_{2,3}(p)-2p\js{-3}-p\js{-2}-\left[\sum_{x(p)}\js{x^3-x^2+x}\right]^2$, where $\delta_{a,b}(p)$ is $1$ if $p$ is $a \bmod b$ and $0$ otherwise.
\end{lemma}

\begin{proof} We have
\begin{eqnarray}
    pA_{2,\mathcal{F}}(p) & \ = \ & \sum_{t(p)}{a_t}^2(p)\nonumber\\
    & = & \sum_{t(p)}\sum_{x(p)}\sum_{y(p)}\js{x^3-tx^2+xt^2-t^2}\js{y^3-ty^2+yt^2-t^2} \nonumber\\
    & = & \sum_{t=1}^{p-1}\sum_{x,y(p)}\js{t^3x^3-t^3x^2+t^3x-t^2}\js{t^3y^3-t^3y^2+t^3y-t^2} \nonumber\\
    & = & \sum_{t=1}^{p-1}\sum_{x,y(p)}\js{t^4}\js{t(x^3-x^2+x)-1}\js{t(y^3-y^2+y)-1} \nonumber\\
    & = & \sum_{t=0}^{p-1}\sum_{x,y(p)}\js{t(x^3-x^2+x)-1}\js{t(y^3-y^2+y)-1}-\sum_{x,y(p)}\js{-1}\js{-1} \nonumber\\
    & = & \sum_{t(p)}\sum_{x,y(p)}\js{t(x^3-x^2+x)-1}\js{t(y^3-y^2+y)-1}-p^2.
\end{eqnarray}

We have a triple sum involving $x, y$ and $t$. If we fix $x$ and $y$ we have a quadratic in $t$. We can determine its contribution by using Lemma \ref{lem:sumlegendrelinquad}; to do so we need to compute the discriminant of the quadratic in $t$, denoted by $\delta(x,y)$:
\begin{eqnarray}
    a & \ = \ & (x^3-x^2+x)(y^3-y^2+y)=xy(x^2-x+1)(y^2-y+1) \nonumber\\
    b & \ = \ & -[(x^3-x^2+x)+(y^3-y^2+y)] \nonumber\\
    c & \ = \ & 1 \nonumber\\
    \delta(x,y) & \ = \ & b^2-4ac = [(x^3-x^2+x)-(y^3-y^2+y)]^2 \nonumber\\
    & \ = \ & \left((x-y)(x^2+xy-x+y^2-y+1)\right)^2 = (\delta_1(x,y)\delta_2(x,y))^2
\end{eqnarray}

Let the vanishing sum be $V(p)$, the contribution when either $x^3-x^2+x$ or $y^3-y^2+y$ or both vanish. We have
\begin{eqnarray}
    pA_{2,\mathcal{F}}(p) & \ = \ & \sum_{t(p)}\sum_{x,y(p)}\js{t(x^3-x^2+x)-1}\js{t(y^3-y^2+y)-1}-p^2 \nonumber\\
    & = & \sum_{t(p)}\sum_{x^3-x^2+x\not\equiv0, \atop y^3-y^2+y \not=0}\js{t(x^3-x^2+x)-1}\js{t(y^3-y^2+y)-1}+V(p)-p^2 \nonumber\\
     & = & \sum_{t(p)}\sum_{x^3-x^2+x, y^3-y^2+y\not\equiv 0 \atop \delta=0}\js{t(x^3-x^2+x)-1}\js{t(y^3-y^2+y)-1} \nonumber\\
  && \ + \  \sum_{t(p)}\sum_{x^3-x^2+x, y^3-y^2+y\not\equiv=0, \atop \delta \not=0}\js{t(x^3-x^2+x)-1}\js{t(y^3-y^2+y)-1}
  \nonumber\\ & & \ \ \ \ \ \ + \ V(p)-p^2 \nonumber\\
     & = & (p-1)\sum_{x^3-x^2+x, y^3-y^2+y \not\equiv=0, \atop \delta=0}\js{(x^3-x^2+x)(y^3-y^2+y)} \nonumber\\
    && \ - \ \sum_{x^3-x^2+x, y^3-y^2+y \not\equiv=0, \atop \delta\not=0}\js{(x^3-x^2+x)(y^3-y^2+y)} + V(p)-p^2 \nonumber\\
     &=& p\sum_{x^3-x^2+x, y^3-y^2+y \not\equiv=0, \atop \delta=0}\js{(x^3-x^2+x)(y^3-y^2+y)} \nonumber\\
    && \ - \ \sum_{x^3-x^2+x\not\equiv0, \atop y^3-y^2+y \not=0}\js{(x^3-x^2+x)(y^3-y^2+y)} + V(p)-p^2.
\end{eqnarray}

We first calculate $V(p)$; we use inclusion/exclusion on $x^3-x^2+x$ and $y^3-y^2+y$ vanishing. Assume first that $x^3-x^2+x$ equals to zero (there are three solutions when $p$ is 1 mod 3: $x=0$, and $x^2-x+1$ has two roots; there is one solution when $p$ is 2 mod 3: $x=0$, and $x^2-x+1$ has no root mod $p$). Then we have $\js{t*0-1}\sum_{t(p)}\sum_{y(p)}\js{t(y^3-y^2+y)-1}=\js{-1}\sum_{t(p)}\sum_{y(p)}\js{t(y^3-y^2+y)-1}$, which is $\js{-1}^2 \cdot 3p =3p$ if $p \equiv 1 \bmod 3$ and $\js{-1}^2 \cdot 1p=1p$ if $p \equiv 2 \bmod 3$ from our $A_{1,\mathcal{F}}(p)$ computation, giving $3 \cdot 3p = 9p$ if $p \equiv 1 \bmod 3$ and $1 \cdot 1p = 1p$ if $p \equiv 2 \bmod 3$.

Now we assume $y^3-y^2+y \equiv 0 \bmod 3$, which similarly gives us $9p$ if $p$ is 1 mod 3, and $p$ if $p$ is 2 mod 3. We subtract the doubly-counted $x^3-x^2+x\equiv y^3-y^2+y \equiv 0$, which equals to $9p$ if $p \equiv 1 \bmod 3$ and $p$ if $p \equiv 2 \bmod 3$. Hence, the contribution from at least one of $x^3-x^2+x$ and $y^3-y^2+y$ vanishing is $9p$ if $p \equiv 1 \bmod 3$ and $p$ if $p \equiv 2 \bmod 3$.

Next, we are going to calculate the contribution of $\sum_{x^3-x^2+x\not\equiv0, \atop y^3-y^2+y \not=0,  \delta=0}\js{(x^3-x^2+x)(y^3-y^2+y)}$. We break into two cases, when each factor of $\delta(x,y)$ is zero modulo $p$, and then deal with the doubly counted pairs.

\bigskip

Case 1: $\delta_1(x,y) \equiv 0 \bmod p$: As $\delta_1(x,y) = x-y$,
we have $x=y$, and the Legendre symbol is 1 as it is a square, so the double-sum is $p-3$ if $p \equiv 1 \bmod 3$ and is $p-1$ if $p \equiv 2 \bmod 3$.

\bigskip

Case 2: $\delta_2(x,y) \equiv 0 \bmod p$: When is $\delta_2(x,y)\ =\ x^2+xy-x+y^2-y+1 = y^2+(x-1)y+(x^2-x+1) \equiv 0(p)$? Using Lemma \ref{lem:quadformulamodp}, we have
\begin{eqnarray}
    y & \ = \ & \frac{-(x-1)\pm\sqrt{(x-1)^2-4(x^2-x+1)}}{2}\nonumber\\
    & = &  \frac{-(x-1)\pm\sqrt{-3x^2+2x-3}}{2},
\end{eqnarray}
which reduces to finding when $-3x^2+2x-3$ is a square mod $p$. We get two values of $y$ if it is equivalent to a non-zero square, one value if it is equivalent to zero, and no value if it is not equivalent to a square. When solving $\delta_2(x,y) \equiv 0(p)$, we make sure such $y \not\in \{ 0,\frac{1 \pm \sqrt{-3}}{2}\}$, as we are assuming $y^3-y^2+y \not\equiv 0(p)$. If $y=0$, we have $x^2-x+1 \equiv 0(p)$ and the solutions are $x=\frac{1 \pm \sqrt{-3}}{2}$. We have excluded this case so we do not need to worry about $y=0$. By symmetry if $y$ is $\frac{1 \pm \sqrt{-3}}{2}$, we have $x=0$ and we can exclude this case too. When we calculate the contribution from $\delta_2(x,y) \equiv 0 \bmod p$ we encounter the following sum: $\sum_{x^2 - x + 1 \equiv 0 \bmod p}\js{-3x^2+2x-3}$, and since $x^2-x+1 \equiv 0 \bmod p$ we will just have $\sum_{x^2 - x + 1 \equiv 0 \bmod p} \js{-3(x^2-x+1)-x}=\sum_{x^2 - x + 1 \equiv 0 \bmod p} \js{-x}$. We show that this sum is $2$ when $p$ is $1 \bmod 3$; it is zero when $p$ is $2 \bmod 3$ as for those primes there are no roots to $x^2 - x + 1 \equiv 0 \bmod p$. Note $-x \equiv x^2 - x + 1 - x \bmod p$, as $x^2 - x + 1 \equiv 0 \bmod p$. Thus $-x \equiv x^2 - 2x + 1 = (x-1)^2 \bmod p$, and $\js{-x} = 1$ for $x$ such that $x^2 - x + 1 \equiv 0 \bmod p$. If $p \equiv 1 \bmod 3$, the contribution in case 2 is

\begin{eqnarray}
    \sum_{x^3-x^2+x \not\equiv 0(p)}\left[1+\js{-3x^2+2x-3}\right]
    & \ = \ & p-3 + \sum_{x^3-x^2+x \not\equiv 0(p)}\js{-3x^2+2x-3} \nonumber\\
    & = & p-3 + \sum_{x(p)}\js{-3x^2+2x-3} \nonumber\\
    && \ - \ \sum_{x=0}\js{-3x^2+2x-3}-\sum_{x^2-x+1\equiv0(p)}\js{-x} \nonumber\\
    & = & p-3 + \sum_{x(p)}\js{-3x^2+2x-3}-\js{-3}\nonumber\\
    && \ - \ \sum_{x^2-x+1\equiv0(p)}\js{-x} \nonumber\\
    & = & p-3+\sum_{x(p)}\js{-3x^2+2x-3}-\js{-3}-2 \nonumber\\
    & \ = \ & p-5-\js{-3}+\sum_{x(p)}\js{-3x^2+2x-3}.
\end{eqnarray}

For such $p \equiv 2 \bmod 3$ we do not have solutions to $x^2 - x + 1 = 0$, so we have fewer terms, and the contribution is
\begin{eqnarray}
    \sum_{x^3-x^2+x \not\equiv 0(p)}\left[1+\js{-3x^2+2x-3}\right]
    & \ = \ & p-1 + \sum_{x^3-x^2+x \not\equiv 0(p)}\js{-3x^2+2x-3} \nonumber\\
    & = & p-1 + \sum_{x(p)}\js{-3x^2+2x-3} \nonumber\\
    && \ - \ \sum_{x=0}\js{-3x^2+2x-3} \nonumber\\
    & = & p-1 + \sum_{x(p)}\js{-3x^2+2x-3}-\js{-3}.
\end{eqnarray}

We use Lemma \ref{lem:quadformulamodp} again. The discriminant now is $2^2-4(-3)(-3)=-32$. Hence, for $p \geq 5$, $p$ does not divide the discriminant, and $\sum_{x(p)}\js{-3x^2+2x-3}$ is $-\js{-3}$.

Thus, if $p \equiv 1 \bmod 3$, for $x \not=0, \frac{1 \pm \sqrt{-3}}{2}$, the number of solutions with $x^2+xy-x+y^2-y\equiv-1$ is $p-5-2\js{-3}$; the number of solutions with $x-y\equiv 0$ is $p-3$. If $p \equiv 2 \bmod 3$, for $x \not=0$, the number of solutions with $x^2+xy-x+y^2-y\equiv-1$ is $p-1-2\js{-3}$; the number of solutions with $x-y\equiv 0$ is $p-1$. To count how many solutions there are of $3x^2 - 2x + 1$, we use the quadratic formula modulo $p$, and note the discriminant is $4 - 12 = -8 = 4(-2)$. So if $-2$ is a square there are two roots, if $-2$ is zero there are no roots and if $-2$ is not a square there are no roots. Thus the number of doubly counted solutions is $1 + \js{-2}$. If $p \equiv 1 \bmod 3$, the total number of pairs is

\begin{eqnarray}
    p-3+p-5-2\js{-3}-1-\js{-2}=2p-9-2\js{-3}-\js{-2},
\end{eqnarray}
and if $p \equiv 2 \bmod 3$, the total number of pairs is
\begin{eqnarray}
    p-1+p-1-2\js{-3}-1-\js{-2}=2p-3-2\js{-3}-\js{-2}.
\end{eqnarray}

We will later see that all the Legendre coefficients multiplying the number of solutions are always $1$, so we just need to count the number of solutions. When $x=y$ (and $x^3-x^2+x\not\equiv0$, $y^3-y^2+y\not\equiv0$), clearly $\js{(x^3-x^2+x)(y^3-y^2+y)}=1$ and these terms each contribute $1$.

Consider $x \not= y$ and $x^2+xy-x+y^2-y+1 \equiv 0$ (and $x^3-x^2+x\not\equiv0$, $y^3-y^2+y\not\equiv0$). Then $x,y\not\equiv0$ and $x^2-x+1\equiv y(-y+1-x)$ and $y^2-y+1\equiv x(-x+1-y)$ and
\begin{eqnarray}
    \js{(x^3-x^2+x)(y^3-y^2+y)} = \js{x^2y^2(-x+1-y)^2}.
\end{eqnarray}
We can see that as long as $x+y \not=1$, all pairs have their Legendre factor $+1$. If $y=1-x$, then we would have $\delta_2(x,y)=y^2+(x-1)y+x^2-x+1 = (1-x)^2+(x-1)(1-x)+x^2-x+1 = x^2-x+1 \equiv 0$, which is impossible as we are assuming in this case that $x^3 - x^2 + x$ is not $0 \bmod p$. Therefore, putting all the pieces together, if $p \equiv 1 \bmod 3$, we have the following form

\begin{eqnarray}
   pA_{2,\mathcal{F}}(p) & \ = \ & p\left(2p-9-2\js{-3}-\js{-2}\right)-\sum_{x,y\not\equiv0, \frac{1 \pm \sqrt{-3}}{2}} \js{(x^3-x^2+x)(y^3-y^2+y)}+9p \nonumber\\
   && \ - \ p^2\nonumber\\
   & = & p^2-2p\js{-3}-p\js{-2}-\left[\sum_{x(p)}\js{x^3-x^2+x}\right]^2,
\end{eqnarray}
and if $p \equiv 2 \bmod 3$, we have the following form
\begin{eqnarray}
   pA_{2,\mathcal{F}}(p) & \ = \ & p\left(2p-3-2\js{-3}-\js{-2}\right)-\sum_{x,y\not\equiv0} \js{(x^3-x^2+x)(y^3-y^2+y)}+p-p^2\nonumber\\
   & = & p^2-2p-2p\js{-3}-p\js{-2}-\left[\sum_{x(p)}\js{x^3-x^2+x}\right]^2.
\end{eqnarray}

\end{proof}

\subsection{Two-Parameter Family}

In the case of our two-parameter family, we unfortunately are going to have cubic and higher degree of Legendre sums; however, by cleverly writing $s$ in terms of $t$, we get linear and quadratic Legendre sums and can calculate them in closed form. Then, we count the number of ways the discriminant of this new closed form vanishes, or equals 0, because these pairs contribute a Legendre factor of +1.

\begin{lemma}
The first moment of the two-parameter family $y^2=x^3+t^2x+st^4$ is $0$.
\end{lemma}

\begin{proof} We have
\begin{eqnarray}
    -p^2A_{1,\mathcal{F}}(p) & \ = \ & -\sum_{t(p)}\sum_{s(p)}a_{t,s}(p) =\sum_{t(p)}\sum_{x(p)}\sum_{s(p)}\js{x^3+t^2x+st^4} \nonumber\\
    & = & \sum_{t=1}^{p-1}\sum_{x(p)}\sum_{s(p)}\js{t^3x^3+t^3x+st^4}\nonumber\\
    & = & \sum_{t=1}^{p-1}\sum_{x(p)}\sum_{s(p)}\js{t^3}\js{x^3+x+st}\nonumber\\
    & = & \sum_{x(p)}\sum_{s(p)}\sum_{t(p)}\js{t}\js{st+(x^3+x)}\nonumber\\
    & = & \sum_{x(p)}\sum_{s(p)}\sum_{t(p)}\js{t}\js{t^{-1}st+(x^3+x)}\nonumber\\
    & = & \sum_{x(p)}\sum_{s(p)}\sum_{t(p)}\js{t}\js{s+(x^3+x)}.
\end{eqnarray} In the proof we used that, at one point, $t$ was not zero. We sent $s$ to $t^{-1} s$, and looked at the resulting $s$ sum, which equals to zero.
\end{proof}

\begin{lemma}
The second moment times $p^2$ of the two-parameter family $y^2=x^3+t^2x+st^4$ is $p^3-2p^2+p-2(p^2-p)\js{-3}$, which supports the bias conjecture.
\end{lemma}

\begin{proof} We have
   \begin{align}
    p^2A_{2,\mathcal{F}}(p) & = \sum_{t,s(p)}{a_{t,s}}^2(p) \nonumber\\
    & =  \sum_{t(p)}\sum_{s(p)}\sum_{x,y(p)}\js{x^3+t^2x+st^4}\js{y^3+t^2y+st^4} \nonumber\\
    &= \sum_{t=1}^{p-1}\sum_{s(p)}\sum_{x,y(p)}\js{t^3x^3+t^3x+st^4}\js{t^3y^3+t^3y+st^4} \nonumber\\
    & =  \sum_{t=1}^{p-1}\sum_{s(p)}\sum_{x,y(p)}\js{t^6}\js{x^3+x+st}\js{y^3+y+st}\nonumber\\
    & =  \sum_{t=0}^{p-1}\sum_{s(p)}\sum_{x,y(p)}\js{x^3+x+st}\js{y^3+y+st}-\sum_{s(p)}\sum_{x,y(p)}\js{x^3+x}\js{y^3+y} \nonumber\\
    & =  \sum_{x,y(p)}\sum_{s(p)}\sum_{t(p)}\js{st+(x^3+x)}\js{st+(y^3+y)}-p\left[\sum_{x(p)}\js{x^3+x}\right]^2 \nonumber\\
    & =  \sum_{s=0}\sum_{t(p)}\left[\sum_{x(p)}\js{x^3+x}\right]^2+\sum_{x,y(p)}\sum_{s\not=0}\sum_{t(p)}\js{st+(x^3+x)}\js{st+(y^3+y)} \nonumber\\ & -p\left[\sum_{x(p)}\js{x^3+x}\right]^2\nonumber\\
    & =  \sum_{x,y(p)}\sum_{s\not=0}\sum_{t(p)}\js{st+(x^3+x)}\js{st+(y^3+y)}\nonumber\\
    & =  \sum_{x,y(p)}\sum_{s\not=0}\sum_{t(p)}\js{ss^{-1}t+(x^3+x)}\js{ss^{-1}t+(y^3+y)}\nonumber\\
    & =  \sum_{x,y(p)}\sum_{s\not=0}\sum_{t(p)}\js{t+(x^3+x)}\js{t+(y^3+y)} \nonumber\\
    & = (p-1)\sum_{x,y(p)}\sum_{t(p)}\js{t+(x^3+x)}\js{t+(y^3+y)},
\end{align} where in passing from the second to the third line we sent $x$ and $y$ modulo $p$ to $tx$ and $ty$, which is valid so long as $t$ is not zero; to keep the sum over all $t$ we need to subtract the $t=0$ contribution. We can also see that when $s=0$, since the $t$-sum is $p$ and there is no $t$ dependence, the contribution from $s=0$ and $t=0$ cancel out each other. Note that now as $s$ is non-zero, we can send $t$ to $s^{-1} t$, and we get a nice quadratic sum in $t$.

We use Lemma \ref{lem:sumlegendrelinquad}. The discriminant of the quadratic in $t$, $\delta(x,y)$, equals
\begin{eqnarray}
    a & \ = \ & 1 \nonumber\\
    b & \ = \ & (x^3+x)+(y^3+y) \nonumber\\
    c & \ = \ & (x^3+x)(y^3+y) \nonumber\\
    \delta (x,y) & \ = \ & b^2-4ac = [(x^3+x)-(y^3+y)]^2\nonumber\\
    & = & [(x-y)(y^2+xy+(1+x^2))]^2,
\end{eqnarray}
and we are going to count the number of ways it vanishes. Therefore,
\begin{align}
    p^2A_{2,\mathcal{F}}(p) & \ = \  (p-1)\left[\sum_{x,y \bmod p \atop \delta(x,y)\equiv 0(p)}\sum_{t(p)}\js{t+(x^3+x)}\js{t+(y^3+y)}\right. \nonumber\\
    &  \left. \ \ \ \ \ \ \ \ \ \ \ \ +  \sum_{x,y \bmod p \atop \delta(x,y)
    \not\equiv0(p)}\sum_{t(p)}\js{t+(x^3+x)}\js{t+(y^3+y)}\right]\nonumber
\end{align}
\begin{align}
    &\ =\ (p-1)\left[\sum_{x,y \bmod p \atop \delta(x,y)\equiv 0(p)}(p-1)+ \sum_{x,y \bmod p \atop \delta(x,y)\not\equiv 0(p)}(-1)\right]
    \nonumber\\
    &\ =\ (p-1)\left[p\sum_{x,y \bmod p \atop \delta(x,y)\equiv 0(p)}+p^2(-1)\right].
\end{align}

We have three cases for $\delta(x,y)\equiv0(p)$.

Case 1: We need to count the number of solutions of $\delta_1(x,y)=x-y\equiv 0$, which happens $p$ times when $x=y$.

\bigskip

Case 2: We need to count the number of solutions of $\delta_2(x,y)=y^2+xy+(1+x^2) \equiv 0$. By the Quadratic Formula mod p, we have
\begin{eqnarray}
    y=\frac{-x \pm \sqrt{-3x^2-4}}{2},
\end{eqnarray}
which reduced to finding when $-3x^2-4$ is a square. Thus, summing over $x$ for $p>2$ yields
\begin{eqnarray}
\sum_{x(p)}\left[1+\js{-3x^2-4}\right] &\ =\ & p + \sum_{x(p)} \js{-3x^2-4} \nonumber\\
&\ =\ & p-\js{-3},
\end{eqnarray}
which follows from Lemma \ref{lem:sumlegendrelinquad}. The discriminant now is $0^2-4 \cdot (-3) \cdot (-4)$. For $p \geq 5$, $p$ does not divide the discriminant, hence this sum is $p-\js{-3}$.

\bigskip

Case 3: We need to be careful and remove the contribution from doubly counted tuples. The double counted pairs satisfy both $x=y$ and $y^2+xy+(1+x^2)\equiv 0(p)$, which means that they satisfy $3x^2+1\equiv0(p)$, or $-3x^2\equiv1$. Thus, there is a double-counted solution if and only if $\js{-3}=1$, and the number of double-counted pairs is $1+\js{-3}$.

\bigskip

Therefore, the total number of pairs for $\delta(x,y)\equiv0(p)$ is
\begin{eqnarray}
   \sum_{\delta_1(x,y)\equiv0}+\sum_{\delta_2(x,y)\equiv0}-\sum_{\delta_1(x,y)\equiv0; \delta_2(x,y)\equiv0} & \ = \ & p+p-\js{-3}-1-\js{-3} \nonumber\\
   & = & 2p-1-2\js{-3}.
\end{eqnarray}

Hence, the second moment times $p^2$ of the family equals
\begin{eqnarray}
    p^2A_{2,\mathcal{F}}(p)
    &\ =\ & (p-1)\left[p\left(2p-1-2\js{-3}\right)+p^2(-1)\right] \nonumber\\
    & = & p(p-1)\left(p-1-2\js{-3}\right) \nonumber\\
    & = & p^3-2p^2+p-2(p^2-p)\js{-3}.
\end{eqnarray}
\end{proof}

\section{Motivation Behind Studying $L$-Functions}
\label{appendix}
As seen in the previous sections, our research revolves around $L$-functions. This is a common theme in mathematics: we can take local data and make a global object, and then deduce behaviors about the local data. As an example, let's look at the famous Fibonacci sequence; this section is included as a brief motivation for the power of generating functions by building on an example hopefully familiar to most readers.

The recurrence relation between Fibonacci numbers is
\begin{eqnarray}
    F_{n+1}\ = \ F_{n}+F_{n-1},
\end{eqnarray}
and the sequence starts with
\begin{eqnarray}
    F_0\ = \ 0, \ \ \ \ F_1\ = \ 1.
\end{eqnarray}

Once we have the recurrence relation and the initial conditions, we can in principle compute every Fibonacci number. However, it is time
consuming: to find $F_n$, we must first find $F_i$ for all $i < n$. Binet's Formula allows us to generate any Fibonacci number, and we can derive Binet's Formula using the following generating function
\begin{eqnarray}
    g(x)\ = \ \sum_{n>0}{F}_nx^n.
\end{eqnarray}

After some algebraic manipulations, we get
\begin{eqnarray}
 \sum_{n\ge 2}F_{n+1}x^{n+1} & \ = \  & \sum_{n\ge 2}F_nx^{n+1}+\sum_{n\ge 2}F_{n-1}x^{n+1}\nonumber\\
 \sum_{n\ge 3}F_{n}x^{n} & \ = \  & \sum_{n\ge 2}F_nx^{n+1}+\sum_{n\ge 1}F_{n}x^{n+2}\nonumber\\
 \sum_{n\ge 3}F_{n}x^{n} & \ = \  & x\sum_{n\ge 2}F_nx^{n}+x^2\sum_{n\ge 1}F_{n}x^{n}\nonumber\\
 g(x)-F_1x-F_2x^{2} & \ = \  & x(g(x)-F_1x)+x^2g(x) \nonumber\\
 g(x) & \ = \  & \frac{x}{1-x-x^2}.
\end{eqnarray}
Although we can expand the above equation using the geometric series formula, that is a poor approach as we would have to then expand $(x+x^2)^n$, and as the two terms are of different degrees in $x$, it would be hard to identify the coefficient of $x$ to a given power. Instead it is better to use the partial fraction expansion obtained by factoring the denominator,
\begin{eqnarray}
g(x)\ =\ \frac{x}{1-x-x^2}\ = \ \frac{1}{\sqrt{5}}\left(\frac{\frac{1+\sqrt{5}}{2}x}{1-\frac{1+\sqrt{5}}{2}x}-\frac{\frac{-1+\sqrt{5}}{2}x}{1-\frac{-1+\sqrt{5}}{2}x}\right).
\end{eqnarray}
Then, using the geometric series formula, we obtain Binet's Formula:
\begin{equation}
    F_n \ = \ \frac{1}{\sqrt{5}}\left[\left(\frac{1+\sqrt{5}}{2}\right)^n-\left(\frac{-1+\sqrt{5}}{2}\right)^n\right],
\end{equation}
which allows us to immediately jump to any Fibonacci number.

\section{Forms of 4th and 6th moments sums}

\subsection{Tools for higher moments calculations}

The Dirichlet Coefficients of the elliptic curve $L$-function can be written as\be a_t(p)\ = \ \sqrt{p}\left(e^{i\theta_t(p)}+e^{-i\theta_t(p)}\right)\ = \ 2\sqrt{p}\cos (\theta_t(p)), \ee with $\theta_t(p)$ real; this expansion exists by Hasse's theorem, which states $|a_t(p)| \le 2 \sqrt{p}$. Define
\be {\rm sym}_k(\theta)\ := \ \frac{\sin((k+1)\theta)}{\sin \theta}.\ee
By the angle addition formula for sine,
\be {\rm sym}_k(\theta)\ = \ {\rm sym}_{k-1}(\theta)\cos \theta+ \cos(k\theta). \ee
When $k=1$, we have
\be {\rm sym}_1(\theta)\ = \ 2\cos \theta. \ee
Michel \cite{Mic} proved that
\be \sum_{t(p)} {\rm sym}_k(\theta_t(p))\ = \ O(\sqrt{p}), \ee where the big-Oh constant depends only on the elliptic curve and $k$; thus while we should have a $k$ subscript in the implied constant, as $k$ is fixed in our investigations we omit it for notational simplicity.

\subsection{Form of 4th moments sums}\label{4th_moment_form}

A lot is known about the moments of the $a_t(p)$ for a fixed elliptic curve $E_t$. However, as we are only concerned with averages over one-parameter families, we do not need to appeal to any results towards the Sato-Tate distribution, and instead we can directly prove convergence of the moments on average to the moments of the semicircle. In particular, the average of the $2m$-th moments has main term $\frac1{m+1}\ncr{2m}{m} p^{m-1}$. The coefficients $\frac1{m+1}\ncr{2m}{m}$ are the Catalan numbers, and the first few main terms of the even moments are $p, 2p^2, 5p^3$ and $14p^4$.

\begin{lemma} The average fourth moment of an elliptic surface with $j(T)$ non-constant has main term $2p^2$: \be \sum_{t(p)}{a_t}^4(p)\ = \  2p^3+O(p^\frac{5}{2}).\ee
\end{lemma}

\begin{proof} We have to compute
\be {a_t}^4(p)\ = \ 16p^2\cos^4 \theta_t(p). \ee We first collect some useful trigonometry identities:
\bea \cos(2\theta)&\ = \ &2\cos^2(\theta)-1\nonumber\\
\cos^2(\theta)&\ = \ &\frac{1}{2}\cos(2\theta)+\frac{1}{2}.\eea

We use these to re-write $\cos^4\theta$ in terms of quantities we can compute:
\bea \cos^4(\theta)&\ = \ &\frac{1}{4}\cos^2(2\theta)+\frac{1}{2}\cos(2\theta)+\frac{1}{4} \nonumber\\
&\ = \ &\frac{1}{8}\cos(4\theta)+\frac{1}{2}\cos(2\theta)+\frac{3}{8}\nonumber\\
&\ = \ &\frac{1}{8}[{\rm sym}_4(\theta)-{\rm sym}_3(\theta)\cos \theta]+\frac{1}{2}\cos(2\theta)+\frac{3}{8}.
\eea

The following expression will arise in our expansion, so we analyze it first:
\bea -\frac{1}{8}{\rm sym}_3(\theta)\cos \theta&\ = \ &-\frac{1}{8}\frac{\sin(4\theta)}{\sin \theta}\cos \theta \nonumber\\
 &\ = \ &-\frac{1}{8}\frac{2\sin(2\theta)\cos(2\theta)}{\sin \theta}\cos \theta \nonumber\\
 &\ = \ &-\frac{1}{8}\frac{2\cdot 2\sin \theta \cos \theta \cos(2\theta)}{\sin \theta}\cos \theta \nonumber\\
 &\ = \ &-\frac{1}{2}\cos^2 \theta \cos(2\theta)\nonumber\\
 &\ = \ &-\frac{1}{2}\cos^2 \theta (2\cos^2 \theta - 1)\nonumber\\
 &\ = \ &-\cos^4 \theta+\frac{1}{2}\cos^2 \theta \nonumber\\
16p^2\cdot \left(-\frac{1}{8}{\rm sym}_3(\theta)\cos \theta\right)&\ = \ &-16p^2\cos^4 \theta+8p^2\cos^2 \theta \nonumber\\
&\ = \ & -16p^2\cos^4 \theta+2p\cdot {a_t}^2(p).
\eea

Thus
\bea
16p^2\cos^4 \theta&\ = \ & 2p^2{\rm sym}_4\theta -16p^2\cos^4 \theta +2p\cdot {a_t}^2(p)+4p\cdot {a_t}^2(p)-2p^2\nonumber\\
2\cdot(16p^2\cos^4 \theta)&\ = \ & 2p^2{\rm sym}_4\theta+ 6p\cdot {a_t}^2(p)-2p^2\nonumber\\
\sum_{t(p)}(16p^2\cos^4 \theta)&\ = \ & p^2\sum_{t_(p)}{\rm sym}_4\theta+ 3p\sum_{t_(p)} {a_t}^2(p)-p^3\nonumber\\
\sum_{t(p)}{a_t}^4(p)&\ = \ &p^2\cdot O(\sqrt{p})+3p(p^2+O(p^\frac{3}{2}))-p^3 \nonumber\\
&\ = \ &2p^3+O(p^\frac{5}{2}),
\eea
as claimed. \end{proof}

\subsection{Form of 6th moments sums}\label{6th_moment_form}

\begin{lemma} The average sixth moment of an elliptic surface with $j(T)$ non-constant has main term $5p^3$: \be \sum_{t(p)}{a_t}^6(p)\ = \  5p^4+O(p^{\frac{7}{2}}).\ee
\end{lemma}

\begin{proof} We have
\bea {a_t}^6(p) &\ = \ & 64p^3\cos^6 \theta_t(p) \nonumber\\
 \cos(3\theta)&\ = \ &4\cos^3 \theta-3\cos \theta \nonumber\\
\cos^3\theta&\ = \ &\frac{\cos(3\theta)+3\cos \theta}{4}.\eea

We first expand $\cos^6\theta$:
\bea
\cos^6\theta&\ = \ &\frac{\cos^2\left(3\theta\right)+9\cos^2\theta+6\cos\theta\cos\left(3\theta\right)}{16} \nonumber\\
&\ = \ &\frac{\frac{1}{2}\cos\left(6\theta\right)+\frac{1}{2}+9[\frac{1}{2}\cos\left(2\theta\right)+\frac{1}{2}]+6\cos\theta[4\cos^3\theta-3\cos\theta]}{16} \nonumber\\
&\ = \ &\frac{10+\cos\left(6\theta\right)+9\cos\left(2\theta\right)+48\cos^4\theta-36\cos^2\theta}{32} \nonumber\\
&\ = \ &\frac{10+\cos\left(6\theta\right)+9\cos\left(2\theta\right)+48[\frac{1}{8}\cos\left(4\theta\right)+\frac{1}{2}\cos\left(2\theta\right)+\frac{3}{8}]-36\cos^2\theta}{32} \nonumber\\
&\ = \ &\frac{10+\cos\left(6\theta\right)+9\cos\left(2\theta\right)+48[\frac{1}{8}\cos\left(4\theta\right)+\frac{1}{2}\cos\left(2\theta\right)+\frac{3}{8}]-18\cos\left(2\theta\right)-18}{32} \nonumber\\
&\ = \ &\frac{10+\cos\left(6\theta\right)+6\cos\left(4\theta\right)+15\cos\left(2\theta\right)}{32} \nonumber\\
&\ = \ &\frac{\cos\left(6\theta\right)}{32}+\frac{10+6\cos\left(4\theta\right)+15\cos\left(2\theta\right)}{32} \nonumber\\
&\ = \ &\frac{{\rm sym}_6\left(\theta\right)-{\rm sym}_5\left(\theta\right)\cos\theta}{32}+\frac{10+6\cos\left(4\theta\right)+15\cos\left(2\theta\right)}{32}.
\eea

Next we find a formula for the symmetric function that will appear:
\bea
-\frac{1}{32}{\rm sym}_5\left(\theta\right)\cos \theta &\ = \ & -\frac{1}{32} \left(\frac{\sin\left(6\theta\right)}{\sin \theta}\right)\cos \theta\nonumber\\
&\ = \ & -\frac{1}{32}\cos \theta \left(\frac{3\sin\left(2\theta\right)-4\sin^3\left(2\theta\right)}{\sin \theta}\right) \nonumber\\
&\ = \ & -\frac{1}{32}\cos \theta \left(\frac{6\sin\theta\cos\theta-32\sin^3\theta\cos^3\theta}{\sin \theta}\right)\nonumber\\
&\ = \ & -\frac{1}{32}\cos \theta [6\cos\theta-32\sin^2\theta\cos^3\theta] \nonumber\\
&\ = \ & -\frac{3}{16}\cos^2\theta +\left(1-\cos^2\theta\right)\cos^4\theta \nonumber\\
&\ = \ & -\frac{3}{16}\cos^2\theta +\cos^4\theta-\cos^6\theta \nonumber\\
64p^3\left(-\frac{1}{32}{\rm sym}_5\left(\theta\right)\cos \theta\right) &\ = \ &-12p^3\cos^2\theta+64p^3\cos^4\theta-64p^3\cos^6\theta \nonumber\\
&\ = \ &-64p^3\cos^6\theta+4p{a_t}^4\left(p\right)-3p^2{a_t}^2\left(p\right).
\eea
Thus
\bea
64p^3\cos^6\theta& = &2p^3{\rm sym}_6\left(\theta\right)-64p^3\cos^6\theta+4p{a_t}^4\left(p\right)-3p^2{a_t}^2\left(p\right) \nonumber\\ & & \ \ \ \ \ \ +\ 12p^3\cos\left(4\theta\right)+30p^3\cos\left(2\theta\right)+20p^3\nonumber\\
64p^3\cos^6\theta& =  &p^3{\rm sym}_6\left(\theta\right)+2p{a_t}^4\left(p\right)-\frac{3}{2}p^2{a_t}^2\left(p\right)+6p^3\cos\left(4\theta\right)+15p^3\cos\left(2\theta\right)+10p^3.
\eea

We can re-express some of the terms above in a more convenient form:
\bea
15p^3\cos\left(2\theta\right)&\ = \ &15p^3\left(2\cos^2\theta-1\right) \nonumber\\
&\ = \ &30p^3\cos^2\theta-15p^3 \nonumber\\
&\ = \ &\frac{15}{2}p^2{a_t}^2\left(p\right)-15p^3
\eea
and
\bea
6p^3\cos\left(4\theta\right)&\ = \ &6p^3[{\rm sym}_4\left(\theta\right)-{\rm sym}_3\left(\theta\right)\cos\theta] \nonumber\\
&\ = \ &6p^3[{\rm sym}_4\left(\theta\right)-8\cos^4\theta+4\cos^2\theta] \nonumber\\
&\ = \ &6p^3{\rm sym}_4\left(\theta\right)-3p{a_t}^4\left(p\right)+6p^2{a_t}^2\left(p\right).
\eea

Thus
\bea
64p^3\cos^6\theta & \ = \ & p^3{\rm sym}_6\left(\theta\right)+2p{a_t}^4\left(p\right)-\frac{3}{2}p^2{a_t}^2\left(p\right)+6p^3{\rm sym}_4\left(\theta\right) -3p{a_t}^4\left(p\right)+6p^2{a_t}^2\left(p\right)\nonumber\\ & & \ \ \ \ \ +\ \frac{15}{2}p^2{a_t}^2\left(p\right)-15p^3+10p^3 \nonumber\\
\sum_{t\left(p\right)}64p^3\cos^6\theta &\ = \  & \sum_{t\left(p\right)}[p^3{\rm sym}_6\left(\theta\right)+2p{a_t}^4\left(p\right)-\frac{3}{2}p^2{a_t}^2\left(p\right)+6p^3{\rm sym}_4\left(\theta\right)-3p{a_t}^4\left(p\right)+6p^2{a_t}^2\left(p\right)\nonumber\\ & & \ \ \ \ \ +\ \frac{15}{2}p^2{a_t}^2\left(p\right)-15p^3+10p^3].
\eea

Therefore
\bea
\sum_{t\left(p\right)}{a_t}^6\left(p\right)&\ = \ &p^3\sum_{t\left(p\right)}{\rm sym}_6\left(\theta\right)+6p^3\sum_{t\left(p\right)}{\rm sym}_4\left(\theta\right)-p\sum_{t\left(p\right)}{a_t}^4\left(p\right)+12p^2\sum_{t\left(p\right)}{a_t}^2\left(p\right)-5p^4 \nonumber\\
&\ = \ & p^3O\left(\sqrt{p}\right)+6p^3O\left(\sqrt{p}\right)-p\left(2p^3+O\left(p^{\frac{5}{2}}\right)\right)+12p^2\left(p^2+O\left(p^{\frac{3}{2}}\right)\right)-5p^4 \nonumber\\
&\ = \ & 5p^4+O\left(p^{\frac{7}{2}}\right),
\eea completing the proof. \end{proof}


{\footnotesize

{\footnotesize
\medskip
\medskip
\vspace*{1mm}

\noindent {\it Steven J. Miller}\\
Williams College \\
880 Main St \\
Williamstown, MA 01267 \\
E-mail: {\tt sjm1@williams.edu, Steven.Miller.MC.96@aya.yale.edu}\\ \\

\noindent {\it Yan Weng}\\
Peddie School\\
201 S Main St\\
Hightstown, NJ 08520\\
E-mail: {\tt yweng-22@peddie.org}\\  \\
}

\vspace*{1mm}\noindent\footnotesize{\date{ {\bf Received}: April 31, 2017\;\;\;{\bf Accepted}: June 31, 2017}}\\
\vspace*{1mm}\noindent\footnotesize{\date{  {\bf Communicated by Some Editor}}}


\begin{thebibliography}{ACFKKLMMWWYY1}

\bibitem[ALM]{ALM}
S. Arms, S. J. Miller and A. Lozano-Robledo, {\em Constructing elliptic curves over $\mathbb{Q(T)}$ with moderate rank}, Journal of Number Theory {\bf 123} (2007), no. 2, 388-402.

\bibitem[ACFKKLMMWWYY]{ACFKKLMMWWYY}
M. Asada, R. Chen, E. Fourakis, Y. Kim, A. Kwon, J. Lichtman, B. Mackall, S. J. Miller, E. Winsor, K. Winsor, J. Yang, and Kevin Yang, \emph{Lower-Order Biases Second Moments of Dirichlet Coefficients in Families of $L$-Functions}, preprint. \bburl{https://web.williams.edu/Mathematics/sjmiller/public_html/math/papers/BiasesinEllipticCurvesPaperFinal1027.pdf}.

\bibitem[BFMT-B]{BFMT-B}
O. Barrett, F. W. K. Firk, S. J. Miller and C. Turnage-butterbaugh, \emph{From Quantum Systems to $L$-Functions: Pair Correlation Statistics and Beyond}, in Open Problems in Mathematics (editors John Nash  Jr. and Michael Th. Rassias), Springer-Verlag, 2016, pages 123--171.

\bibitem[BAU]{BAU}
L. Bauer, {\em Weierstrass equations: Seminar on elliptic curves and the Weil conjectures}, to appear in the 4th talk in the seminar on elliptic curves and the Weil conjectures supervised by Prof. Dr. Moritz Kerz in the summer term at the University of Regensburg (2016), \bburl{http://www.mathematik.uni-regensburg.de/kerz/ss16/ausarb/bauer.pdf}.

\bibitem[BEW]{BEW}
B. Berndt, R. Evans, and K. Williams, \emph{Gauss and Jacobi Sums}, Canadian Mathematical Society Series of Monographs and Advanced Texts, Vol. 21, 1998.

\bibitem[Bi]{Bi}
B. J. Birch, \emph{How the number of points of an elliptic curve over a fixed prime field varies}, J. London Math, Soc. {\bf 43} (1968), 57-60.

\bibitem[Da1]{Da1}
H. Davenport, \emph{The Higher Arithmetic: An Introduction to the
Theory of Numbers}, 7th edition, Cambridge University Press,
Cambridge, 1999.

\bibitem[Da2]{Da2}
H. Davenport, \emph{Multiplicative Number Theory}, 2nd edition,
revised by H. Montgomery, Graduate Texts in Mathematics, Vol. 74,
Springer-Verlag, New York, 1980.

\bibitem[Du]{Du}
A. Dujella, \emph{History of elliptic curves rank records},
\bburl{https://web.math.pmf.unizg.hr/~duje/tors/rankhist.html}.

\bibitem[FM]{FM}
F. W. K. Firk and S. J. Miller, \emph{Nuclei, Primes and the Random Matrix Connection}, Symmetry \textbf{1} (2009), 64--105; doi:10.3390/sym1010064. \texttt{http://www.mdpi.com/2073-8994/1/1/64}.

\bibitem[HKLL-RM]{HKLL-RM}
T. Hammonds, S. Kim, B. Logsdon, A. Lozano-Robledo and S. J. Miller, \emph{Rank and Bias in Families of Hyperelliptic Curves via Nagao's Conjecture}, to appear in the Journal of Number Theory.

\bibitem[KS2]{KS2}
N. Katz and P. Sarnak, \emph{Zeros of zeta functions and
symmetries}, Bull. AMS \textbf{36} (1999), 1--26.

\bibitem[KN1]{KN1}
M. Kazalicki and B. Naskrecki, \emph{Second moments and the bias conjecture for the family of cubic pencils} (2020), preprint, \bburl{https://arxiv.org/pdf/2012.11306}.

\bibitem[KN2]{KN2}
M. Kazalicki and B. Naskrecki, \emph{Diophantine triples and K3 surfaces} (2021), preprint, \bburl{https://arxiv.org/pdf/2101.11705}.

\bibitem[KS1]{KS1}
N. Katz and P. Sarnak, \emph{Random Matrices, Frobenius Eigenvalues
and Monodromy}, AMS Colloquium Publications, Vol. 45, AMS,
Providence, RI, 1999.

\bibitem[Kn]{Kn}
A. Knapp, \emph{Elliptic Curves}, Princeton University Press,
Princeton, NJ, 1992.

\bibitem[MMRW]{MMRW}
B. Mackall, S. J. Miller, C. Rapti and K. Winsor, \emph{Lower-Order Biases in Elliptic Curve Fourier Coefficients in Families}, to appear in the Conference Proceedings of the Workshop on Frobenius distributions of curves at CIRM in February 2014.

\bibitem[MMRSY]{MMRSY}
D. Mehrle, S. J. Miller, T. Reiter, J. Stahl and D. Yott, \emph{Constructing families of moderate-rank elliptic curves over number fields}, Minnesota Journal of Undergraduate Mathematics \textbf{2} (2016--2017), 11 pages.

\bibitem[Mic]{Mic}
P. Michel, {\em Rang moyen de famille de courbes elliptiques et lois de Sato-Tate}, Monatshefte fur Mathematik {\bf 120} (1995), 127--136.

\bibitem[Mi1]{Mi1}
S. J. Miller, {\em 1- and 2-level densities for families of elliptic curves: evidence for the underlying group symmetries}, Princeton University, PhD thesis (2002). \bburl{http://web.williams.edu/Mathematics/sjmiller/public_html/math/thesis/SJMthesis_Rev2005.pdf}.

\bibitem[Mi2]{Mi2}
S. J. Miller, {\em 1- and 2-level densities for families of elliptic curves: evidence for the underlying group symmetries}, Compositio Mathematica {\bf 140} (2004), no. 4, 952--992.

\bibitem[Mi3]{Mi3}
S. J. Miller, {\em Variation in the number of points on elliptic curves and applications to excess rank}, C. R. Math. Rep. Acad. Sci. Canada {\bf 27} (2005), no. 4, 111--120.

\bibitem[MT-B]{MT-B}
S. J. Miller and R. Takloo-Bighash, \emph{An Invitation to Modern Number Theory}, Princeton University Press, 2006.

\bibitem[MWe]{MWe}
S. J. Miller and Y. Weng, \emph{Biases in Moments of Dirichlet Coefficients of Elliptic Curve Families} (2021),  \bburl{http://arxiv.org/abs/2102.02702}.

\bibitem[Na1]{Na1}
K. Nagao, \emph{Construction of high-rank elliptic curves}, Kobe J. Math. \textbf{11} (1994), 211--219.

\bibitem[Na2]{Na2}
K. Nagao, \emph{$\mathbb{Q}(T)$-rank of elliptic curves and certain limit coming from the local points}, Manuscr. Math. \textbf{92}
(1997), 13--32.

\bibitem[NZM]{NZM}
I. Niven, H. Zuckerman, and H. Montgomery, \emph{An Introduction
to the Theory of Numbers}, 5th edition, John Wiley \& Sons, New
York, 1991.

\bibitem[RG]{RG}
R. Rivest as the lecturer and D. Ghosh as the scribe, {\em 6.857 Computer and Network Security, Lecture 8}, \bburl{http://web.mit.edu/6.857/OldStuff/Fall97/lectures/lecture8.pdf}.

\bibitem[RoSi]{RoSi}
M. Rosen and J. Silverman, {\em On the rank of an elliptic surface}, Invent. Math. {\bf 133} (1998), 43--67.

\bibitem[Si0]{Si0}
J. Silverman, {\em An Introduction to the Theory of Elliptic Curves}, to appear in the Summer School on \textit{Computational Number Theory and Applications to Cryptography} at University of Wyoming in July 2006. \bburl{https://www.math.brown.edu/~jhs/Presentations/WyomingEllipticCurve.pdf}.

\bibitem[ST]{ST}
J. Silverman and J. Tate, \emph{Rational Points on Elliptic
Curves}, Springer-Verlag, New York, 1992.

\bibitem[Su]{mit_notes}
A. Sutherland, \emph{Point Counting}, \bburl{https://ocw.mit.edu/courses/mathematics/18-783-elliptic-curves-spring-2015/lecture-notes/MIT18_783S15_lec8.pdf}.

\bibitem[Var]{Var}
A. Varilly, \emph{Dirichlet's Theorem on Arithmetic Progressions}, \bburl{https://math.rice.edu/~av15/Files/Dirichlet.pdf}.

\bibitem[WAZ]{WAZ}
R. Wazir, {\em Arithmetic on elliptic threefolds}, Composito Mathematica \textbf{140} (2004), 567-580.

\bibitem[Wu]{Wu}
J. Wu, \emph{Biases in First and Second Moments of the Dirichlet Coefficients in One- and Two-Parameter Families of Elliptic Curves} (student research project advised by S. J. Miller),  \bburl{https://web.williams.edu/Mathematics/sjmiller/public_html/math/papers/WuBiasesinEllipticCurvesPaperFinal81.pdf}.


\end{thebibliography}
\end{document}